\documentclass[11pt]{article}
\usepackage[letterpaper,margin=1in]{geometry}
\usepackage[T1]{fontenc}
\usepackage{lmodern}
\usepackage{amsmath,amssymb,amsthm,mathtools,microtype,xcolor}
\usepackage{tcolorbox}
\tcbuselibrary{breakable}
\usepackage{hyperref}
\hypersetup{hidelinks,pdftitle={An inverse-linear S3 metric on the Gromoll--Meyer sphere: a detailed argument for review}}
\newtcolorbox{proofgap}[1]{breakable,colback=white,
	colframe=red!75!black,coltitle=red!75!black,colbacktitle=white,
	fonttitle=\bfseries,title={#1},boxrule=1pt,sharp corners,
	left=7pt,right=7pt,top=6pt,bottom=6pt,before skip=12pt,after skip=12pt}
\newtheorem{theorem}{Theorem}[section]
\newtheorem{lemma}[theorem]{Lemma}
\newtheorem{proposition}[theorem]{Proposition}

\newtheorem{remark}[theorem]{Remark}

\newcommand{\Sp}{\operatorname{Sp}}
\newcommand{\Ad}{\operatorname{Ad}}

\newcommand{\diag}{\operatorname{diag}}
\newcommand{\HH}{\mathbb H}
\newcommand{\R}{\mathbb R}

\newcommand{\C}{\mathbb C}
\newcommand{\im}{\operatorname{Im}}
\newcommand{\Rea}{\operatorname{Re}}
\newcommand{\tr}{\operatorname{tr}}
\newcommand{\Span}{\operatorname{span}}
\newcommand{\Id}{\operatorname{Id}}

\numberwithin{equation}{section}
\title{A positively curved metric on the Gromoll--Meyer sphere}
\author{Zexuan Ouyang}
\date{}

\begin{document}
	
	\bibliographystyle{plain}
	
	\maketitle
	
	\begin{abstract}
		We construct an explicit deformation of the bi-invariant
		metric on \(\Sp(2)\) whose induced metrics on the Gromoll--Meyer exotic 7-sphere have positive sectional curvature for all sufficiently small positive
		parameters. The deformation is determined by a fixed self-adjoint
		operator compatible with the quotient action. The proof combines a
		uniform curvature estimate with an algebraic analysis of the horizontal
		commuting planes. On the critical set where the lower-order positive
		terms vanish, we establish an explicit positive lower bound for the
		cubic coefficient of the curvature numerator. Uniform control of the
		fourth-order remainder and a compactness argument then show that every sectional curvature is positive.
	\end{abstract}

	\setcounter{tocdepth}{1}

	\section{Introduction}
	\label{sec:introduction}

	The construction of metrics with positive sectional curvature on exotic
	spheres is a basic problem concerning the interaction between curvature
	and smooth structure. 
	The question of whether there exists a metric with positive curvature on an exotic sphere was posed as the second problem in Yau’s Problem Section \cite{yau-problem}. The Gromoll--Meyer sphere provides a concrete
	setting for this problem. Gromoll and Meyer realized this exotic
	\(7\)-sphere as a biquotient of \(\Sp(2)\) and equipped it with a metric
	of nonnegative sectional curvature \cite{GM}. Wilhelm \cite{Wilh} and
	Eschenburg--Kerin \cite{EK} subsequently constructed almost-positive
	metrics, for which all sectional curvatures are positive on an open
	dense subset.
	
	More generally, Goette, Kerin, and Shankar proved that every exotic
	\(7\)-sphere admits an \(\mathrm{SO}(3)\)-invariant metric of nonnegative
	sectional curvature \cite[Theorems~A--C and Corollary~D]{GKS}. Their
	construction uses a six-parameter family of quotients of
	cohomogeneity-one manifolds, together with a calculation of the
	Eells--Kuiper invariant that identifies the resulting exotic smooth
	structures. For positive sectional curvature on the Gromoll--Meyer
	sphere, Petersen--Wilhelm announced a construction in \cite{PW} and
	developed related deformation principles in \cite{PWprinciples}.
	These works are direct predecessors of the present paper.
	
	In this paper, we study a specific inverse-linear deformation on
	\(\Sp(2)\) and prove positivity of its quotient metrics by explicit
	curvature estimates. The inverse-linear formalism and its curvature
	variation formulas originate in the work of Huizenga--Tapp \cite{HT}.
	Our contribution is the choice of a fixed deformation operator and the
	verification of positivity for that operator, including uniform control
	over all tangent planes. We make no priority claim for the underlying
	methods or for the existence of a positively curved metric on the
	Gromoll--Meyer sphere. The quotient construction and the operator
	formalism are recalled with their attributions in
	Sections~\ref{sec:metric} and~\ref{sec:quotient-metric-framework};
	Section~\ref{sec} specifies the deformation considered here.
	
	To state the result, let \(\Sigma\) be the quotient of \(\Sp(2)\) by
	the free action
	$$
	q\star g=\diag(q,q)g\diag(q,1)^{-1},
	\qquad q\in\Sp(1),\quad g\in\Sp(2).
	$$
	On \(\mathfrak{sp}(2)\), fix the bi-invariant inner product
	$$
	B(U,V)=-\frac12\Rea\tr_{\HH}(UV).
	$$
	Write its elements in the coordinates
	$$
	(r;u_0,u_1,u_2)
	=\begin{pmatrix}
		u_0+u_1&r+u_2\\
		-r+u_2&u_0-u_1
	\end{pmatrix},
	\qquad r\in\R,\quad u_0,u_1,u_2\in\im\HH,
	$$
	and define the \(B\)-self-adjoint operator
	$$
	S_3(r;u_0,u_1,u_2)=(3r;u_1,u_0-u_2,-u_1).
	$$
	For sufficiently small \(\varepsilon\ge0\), the right-invariant metric
	$$
	\widetilde g_{\varepsilon,g}(Ug,Vg)
	=B\bigl((I+\varepsilon S_3)^{-1}U,V\bigr)
	$$
	is invariant under the quotient action and therefore induces a smooth
	metric \(g_\varepsilon\) on \(\Sigma\).
	
	\begin{theorem}
		\label{thm:introduction}
		There exists \(\varepsilon_0>0\) such that the quotient metric
		\(g_\varepsilon\) has positive sectional curvature for every
		\(0<\varepsilon\le\varepsilon_0\).
	\end{theorem}
	
	The main difficulty is to control the curvature uniformly near the
	zero-curvature planes of the undeformed quotient metric. In the
	bi-invariant metric, these planes are represented by horizontal
	commuting pairs in \(\mathfrak{sp}(2)\). Under the deformation, the
	O'Neill term provides a positive quadratic contribution on many such
	pairs. The remaining pairs form a critical set on which positivity must
	be detected at cubic order. Establishing positivity along each fixed
	plane is only part of the argument: nearby planes must also be
	controlled on a common parameter interval.
	
	We address this issue by parametrizing all deformed horizontal planes
	using a fixed compact space \(\mathcal I\) of undeformed horizontal
	\(B\)-orthonormal frames. If \(N_\varepsilon\) denotes the quotient
	curvature numerator, we prove a uniform estimate
	$$
	N_\varepsilon
	\ge a_0\bigl(|[x,y]|^2+\varepsilon^2|\zeta(x,y)|^2\bigr)
	+\varepsilon^3\delta_3(x,y)-C_0\varepsilon^4,
	$$
	where \(a_0,C_0>0\) are independent of the frame. Here, on commuting pairs, \(\zeta\)
	measures the first-order contribution to the vertical bracket, and
	\(\delta_3\) is the cubic coefficient of the ambient curvature on
	commuting pairs; their precise definitions appear in
	Section~\ref{sec:roadmap}. On the critical set
	$$
	\mathcal Z=\{(g,x,y)\in\mathcal I:
	[x,y]=0,\ \zeta(x,y)=0\},
	$$
	we establish the bound
	$$
	\delta_3(x,y)\ge\frac{\sqrt5-1}{4}>0.
	$$
	The algebraic argument begins with a restriction imposed by
	horizontality, then reduces the commuting and critical equations to
	scalar identities from which this bound follows. Compactness separates
	the region where the cubic coefficient is small from \(\mathcal Z\).
	Combining this separation with the uniform estimate proves positivity
	on one parameter interval for every plane. The endpoint is specified
	using a positive compact minimum; no numerical value of that minimum
	is required.
	
	After recalling the quotient geometry and deformation formalism, we
	construct the metric in Section~\ref{sec} and formulate the two main
	estimates in Section~\ref{sec:roadmap}. The uniform curvature estimate
	is proved in Section~\ref{s3:sec:analytic}. The horizontal restriction
	and the algebraic critical inequality are established in
	Sections~\ref{s3:detail:spectral} and~\ref{s3:hand:critical},
	respectively. Section~\ref{rw:completion} combines these ingredients
	and proves Theorem~\ref{thm:introduction}, equivalently
	Theorem~\ref{thm:main}.

	\medskip
	\noindent
	{\bf Acknowledgements.}
	The author is  supported by the National Key R\&D Program of China 2023YFA1009900 and NSFC No. 12401073.

	\medskip
	\noindent
	{\bf AI usage.}
	The main result was proved using the Rethlas agent with the GPT-5.6 and GPT-6 models. The author subsequently simplified the proof and improved the presentation, and is fully responsible for all assertions in this paper. The author also thanks Professors Hanlong Fang and Xin Fu for their help in checking the proof. 
	
	\section{Gromoll--Meyer Sphere} \label{sec:metric}
	
	We first recall the construction and conventions of Gromoll and Meyer \cite{GM}. 
	
	Every quaternion has a unique expression
	$q=a+bi+cj+dk$, where $a,b,c,d\in\R$ and
	\[
	i^2=j^2=k^2=-1,\qquad
	ij=k,\quad jk=i,\quad ki=j,\qquad
	ji=-k,\quad kj=-i,\quad ik=-j.
	\]
	Its conjugate and squared norm are
	\[
	\bar q=a-bi-cj-dk,\qquad
	|q|^2=q\bar q=a^2+b^2+c^2+d^2.
	\]
	We can identify the imaginary quaternions
	$\operatorname{Im}\HH$ with $\R^3$ and define the Euclidean inner product on $\operatorname{Im}\HH$ by
	\[
	\langle u,v\rangle_{\mathrm E}
	:=
	\operatorname{Re}(\bar u v)
	=
	-\operatorname{Re}(uv),
	\qquad u,v\in\operatorname{Im}\HH.
	\]
	Explicitly, for $u=u_1i+u_2j+u_3k$ and
	$v=v_1i+v_2j+v_3k$,
	\begin{equation}\label{EN}
		\langle u,v\rangle_{\mathrm E}=u_1v_1+u_2v_2+u_3v_3.    
	\end{equation}
	We also write $u\cdot v=\langle u,v\rangle_{\mathrm E}$
	and denote by $u\times v$ the Euclidean cross product
	with respect to the orientation determined by $(i,j,k)$.
	Quaternion multiplication then satisfies
	\[
	uv=-u\cdot v+u\times v,
	\qquad u,v\in\operatorname{Im}\HH.
	\]

	Define the compact symplectic group by
	\[
	\Sp(n)\coloneq
	\left\{
	Q\in M_n(\HH):
	QQ^*=Q^*Q=\Id
	\right\},
	\]
	where $Q^*$ denotes quaternionic conjugate transpose.  In particular,
	\[
	\Sp(1)=\{q\in\HH:|q|=1\}=S^3.
	\]
	
	Consider the free action of
	\[
	G_0=S^3\times S^3=\Sp(1)\times\Sp(1)
	\]
	on $\Sp(2)$ defined by
	\begin{equation}
		\label{eq:GM-full-action}
		(q_1,q_2)\cdot Q
		=
		\begin{pmatrix}
			q_1&0\\
			0&q_1
		\end{pmatrix}
		Q
		\begin{pmatrix}
			\bar q_2&0\\
			0&1
		\end{pmatrix}.
	\end{equation}
	The quotient of $\Sp(2)$ by the second factor can be identified with the standard $S^7$:
	\[
	\Sp(2)/(\{1\}\times S^3)\longrightarrow S^7\subset\HH^2,
	\qquad
	\begin{pmatrix}a&b\\c&d\end{pmatrix}
	\longmapsto (b,d).
	\]
	
	\begin{lemma}
		\label{lem:unitary-completions}
		Let $\HH^2$ be regarded as a right quaternionic vector space, equipped
		with the standard quaternionic Hermitian inner product
		\[
		\langle x,y\rangle_{\HH}=x^*y.
		\]
		Fix a unit vector $v\in\HH^2$. Then its orthogonal complement
		\[
		v^\perp:=\{x\in\HH^2:v^*x=0\}
		\]
		has quaternionic dimension one. If $u_0\in v^\perp$ is any fixed unit
		vector, then every unit vector $u\in v^\perp$ has a unique
		representation $u=u_0q$, $q\in\Sp(1)$.
	\end{lemma}
	
	\begin{proof}
		It is easy to verify the orthogonal decomposition $\HH^2=v\HH\oplus v^\perp$.
		Since $\HH^2$ has quaternionic dimension two and $v\HH$ has
		quaternionic dimension one, we have $\dim_{\HH}v^\perp=1$.
		
		Choose a unit vector $u_0\in v^\perp$. Since $v^\perp$ is a
		one-dimensional right quaternionic vector space, every $u\in v^\perp$
		can be written uniquely as
		$u=u_0q$ for some $q\in\HH$, where $q=u_0^*u$. In particular, 
		$u$ is a unit vector if and only if $|q|=1$.
		
		Write a quaternionic matrix by its columns as $Q=(u\;v)$.
		Since $v$ is already a unit vector, the equation $Q^*Q=I$ is
		equivalent to $u^*u=1$, $v^*u=0$. Thus the matrices in $\Sp(2)$ whose second column is $v$ are precisely $Q_q=(u_0q\;v)$, $q\in\Sp(1)$.
	\end{proof}

	Now consider the diagonal subgroup
	\[
	\Delta \coloneq
	\{(q,q):q\in S^3\}
	\subset S^3\times S^3.
	\]
	Equivalently, the restricted action of $\Delta$ is
	\begin{equation}
		\label{eq:GM-diagonal-action}
		q\star Q \coloneq
		\begin{pmatrix}
			q&0\\
			0&q
		\end{pmatrix}
		Q
		\begin{pmatrix}
			\bar q&0\\
			0&1
		\end{pmatrix}.
	\end{equation}
	The Gromoll--Meyer manifold is the diagonal orbit space
	\begin{equation}
		\label{eq:GM-sphere-definition}
		\Sigma^7=\Sp(2)/\Delta.
	\end{equation}
	The inclusions of the orbit spaces give an $S^3$-bundle
	\begin{equation}
		\label{eq:GM-bundle}
		S^3
		\longrightarrow
		\Sigma^7
		\longrightarrow
		\Sp(2)/(S^3\times S^3)\cong S^4.
	\end{equation}
	Here, the diffeomorphism $\Sp(2)/(S^3\times S^3)\longrightarrow
	S^4\subset\HH\oplus\R$ is explicitly given by
	\begin{equation*}
		\begin{aligned}
			\begin{pmatrix}a&b\\c&d\end{pmatrix}
			&\longmapsto
			\bigl(2\bar b d,\ |b|^2-|d|^2\bigr).
		\end{aligned}
	\end{equation*}
	
	To identify the particular
	exotic sphere structure, we define two charts $h_1,h_2:\R^4\times S^3\rightarrow\Sigma^7$ by
	\begin{equation*}
		h_1(u,q):=
		\left[
		\frac{1}{(1+|u|^2)^{1/2}}
		\begin{pmatrix}
			q&\bar u\\
			-uq&1
		\end{pmatrix}
		\right]_{\Delta}\,, \,\,\,\,\,\,\,\,\,
		h_2(v,r):=
		\left[
		\frac{1}{(1+|v|^2)^{1/2}}
		\begin{pmatrix}
			\bar v r&1\\
			-r&v
		\end{pmatrix}
		\right]_{\Delta}\,,
	\end{equation*}
	where $[Q]_{\Delta}$ denotes the $\Delta$-orbit of $Q$.
	The images of $h_1$ and $h_2$ are the sets of matrices with $d\ne0$ and $b\ne0$, respectively.  On their overlap the transition map is
	\begin{equation*}
		h_2^{-1}h_1(u,q)
		=
		\left(
		u|u|^{-2},
		\left(\frac{u}{|u|}\right)^2
		q
		\left(\frac{u}{|u|}\right)^{-1}
		\right).
	\end{equation*}
	This is precisely Milnor's clutching map of type \((2,-1)\).  Hence $\Sigma^7\cong M_{3}$ in Milnor's one-index notation \cite{Milnor56,GM}. The same smooth manifold has a description as a Brieskorn--Pham link:
	\[
	W_{5,3}^7
	=
	\left\{
	(u,z_0,z_1,z_2,z_3)\in S^9\subset\mathbb C^5:
	u^5+z_0^3+z_1^2+z_2^2+z_3^2=0
	\right\}.
	\]
	Brieskorn proved that this link is a homotopy \(7\)-sphere and that its
	parallelizable Milnor fiber has signature \(8\)
	\cite{Brieskorn1966}.  Equivalently, it represents a generator of
	\(\Theta_7\cong\mathbb Z/28\mathbb Z\).

	\subsection{Gromoll--Meyer metric with nonnegative sectional curvature}

	We regard \(\Sp(2)\) as a matrix group in \(M_2(\HH)\), with Lie algebra
	\[\mathfrak{sp}(2) \coloneq
	\{A\in M_2(\HH):A^*+A=0\}= \left\{
	\begin{pmatrix}
		a&b\\-\bar b&c
	\end{pmatrix}
	:
	a,c\in\operatorname{Im}\HH,\ b\in\HH
	\right\}.
	\]
	Following the notation of \cite{GM}, for \(Q\in \Sp(2)\), let
	\[
	R_Q:\Sp(2)\longrightarrow \Sp(2),
	\qquad R_Q(A)=AQ
	\]
	denote right translation.
	Write its differential at the identity as
	\[R_{Q*}:=(dR_Q)_I:
	\mathfrak{sp}(2)\longrightarrow T_Q\Sp(2).
	\]  Then, the right-invariant vector field $X^R$ determined by
	\(X\in\mathfrak{sp}(2)\) satisfies
	\begin{equation*}
		(X^R)_Q=R_{Q*}X.  
	\end{equation*}
	The corresponding flow is \begin{equation*}
		Q\longmapsto e^{tX}Q.   
	\end{equation*} 
	For
	$X\in\mathfrak{sp}(2)$, the curve $e^{tX}$ has
	initial velocity $X$, so
	\[
	R_{Q*}X
	=
	\left.\frac{d}{dt}\right|_{t=0}R_Q(e^{tX})
	=
	\left.\frac{d}{dt}\right|_{t=0}e^{tX}Q
	=
	XQ.
	\]
	Here the expression \(XQ\) is the ordinary product of quaternionic matrices representing tangent vectors. Note that with the vector-field bracket convention $[\mathcal X,\mathcal Y]f=\mathcal X(\mathcal Yf)-\mathcal Y(\mathcal Xf)$, one obtains
	\begin{equation}\label{rlie} 
		[X^R,Y^R]=-[X,Y]^R.
	\end{equation}

	Throughout, we use the right trivialization
	\begin{equation}\label{rightt}
		T_QG\longrightarrow\mathfrak{sp}(2),
		\qquad V\longmapsto VQ^{-1},
	\end{equation}
	which is the inverse of $R_{Q*}$. Thus every
	$V\in T_QG$ can be written uniquely as
	\[
	V=R_{Q*}X=XQ,
	\qquad X=VQ^{-1}\in\mathfrak{sp}(2).
	\]
	The left-trivialized coordinate of the same vector
	is $Q^{-1}V$.

	For computational convenience, we introduce
	\[
	m:=
	\begin{pmatrix}0&1\\-1&0\end{pmatrix},
	\,\,\,\,\,
	U_0(u):=
	\begin{pmatrix}u&0\\0&u\end{pmatrix},
	\,\,\,\,\,
	U_1(u):=
	\begin{pmatrix}u&0\\0&-u\end{pmatrix},
	\,\,\,\,\,
	U_2(u):=
	\begin{pmatrix}0&u\\u&0\end{pmatrix},\,\,\,u\in\operatorname{Im}\mathbb H.
	\]
	Then, each \(Z\in\mathfrak{sp}(2)\) has a unique expression
	\begin{equation}
		\label{eq:multiplicity-coordinates} Z:=rm+\sum\limits_{a=0}^2U_a(u_a) =
		\begin{pmatrix}
			u_0+u_1&r+u_2\\
			-r+u_2&u_0-u_1
		\end{pmatrix}=:(r;u_0,u_1,u_2),\,\,\,r\in\mathbb R, \,\,\,u_0,u_1,u_2\in\operatorname{Im}\mathbb H.
	\end{equation}
	We call $(r;u_0,u_1,u_2)$ the multiplicity coordinates of $Z$.
	For \(Z=(r;u_0,u_1,u_2)\) and
	\(W=(s;v_0,v_1,v_2)\), writing $[Z,W]=2(\rho;w_0,w_1,w_2)$, we have
	\begin{equation}
		\label{eq:multiplicity-bracket}
		\begin{aligned}
			\rho&=u_2\cdot v_1-u_1\cdot v_2,\,\,\,\,\,\,\,\,\,\,\,\,\,\,\,\,\,\,\,\,\,\,\,\,\,\,\,\,\,\,\,\,\,\,\,\,\,\,\,\, w_0=u_0\times v_0+u_1\times v_1+u_2\times v_2,\\
			w_1&=u_0\times v_1+u_1\times v_0+rv_2-su_2,\,\,
			w_2=u_0\times v_2+u_2\times v_0+su_1-rv_1.
		\end{aligned}
	\end{equation}

	Taking the trace on \(\mathfrak{sp}(2)\) gives the negative-definite Killing form
	\begin{equation*}
		\kappa(X,Y)=\operatorname{tr}_{\mathbb R} \bigl(\operatorname{ad}_X\circ\operatorname{ad}_Y\bigr), 
	\end{equation*} where $\operatorname{ad}_X:\mathfrak{sp}(2)\longrightarrow\mathfrak{sp}(2)$, $\operatorname{ad}_X(Y)=[X,Y]$ is the adjoint representation of the Lie algebra. For the quaternionic matrix realization of \(\mathfrak{sp}(2)\), computation yields that $\kappa(X,Y)=12\operatorname{Re}\operatorname{tr}_{\HH}(XY)$. We choose the normalization
	\begin{equation}
		\label{BN}
		B(X,Y):=-\frac1{24}\kappa(X,Y)=-\frac12\operatorname{Re}\operatorname{tr}_{\HH}(XY),
	\end{equation}
	so that, in the coordinates \eqref{eq:multiplicity-coordinates}, \begin{equation}\label{bp}
		B(Z,W)=rs+u_0\cdot v_0+u_1\cdot v_1+u_2\cdot v_2,  
	\end{equation}
	and the elements $\{m, U_a(i), U_a(j),U_a(k),a=0,1,2\}$ form a \(B\)-orthonormal basis. 
	We define a bi-invariant metric $\widetilde g_0$ on \(\Sp(2)\) such that
	\begin{equation}
		\label{eq:GM-ambient-metric}
		(\widetilde g_{0})_{Q}(XQ,YQ):=B(X,Y).
	\end{equation}
	
	For a Riemannian manifold \((M,g)\) with the Levi--Civita connection \(\nabla\), we use
	the curvature convention
	\begin{equation}\label{curvconv}
		{\mathcal R}(\mathcal X,\mathcal Y)\mathcal Z \coloneq \nabla_{\mathcal X}\nabla_{\mathcal Y}\mathcal Z
		- \nabla_{\mathcal Y}\nabla_{\mathcal X}\mathcal Z - \nabla_{[\mathcal X,\mathcal Y]}\mathcal Z.    
	\end{equation}
	For linearly independent vectors \(v,w\in T_pM\), the
	sectional curvature of the plane they span is \begin{equation}
		\label{eq:sectional-curvature}
		\sec(v\wedge w)
		=
		\frac{g_p(\mathcal R(v,w)w,v)}
		{g_p(v,v)g_p(w,w)-(g_p(v,w))^2}=:\frac{\mathcal N^{g}(v,w)}{g_p(v,v)g_p(w,w)-(g_p(v,w))^2}.
	\end{equation}
	Then, the curvature formula for a bi-invariant metric
	\cite[p.~224]{Lee} yields that
	\begin{equation}
		\label{eq:bi-invariant-curvature}
		\left(\widetilde g_{0})_{Q}(\mathcal R(XQ,YQ)(YQ),
		XQ
		\right)= \frac14\|[X,Y]\|_B^2,
	\end{equation}
	where $\|Z\|_B^2=B(Z,Z)$.
	It follows from \eqref{eq:sectional-curvature} and
	\eqref{eq:bi-invariant-curvature} that \(\widetilde g_0\)
	has nonnegative sectional curvature.  Moreover, for linearly independent \(X,Y\in\mathfrak g\), the ambient two-plane
	spanned by \(XQ\) and \(YQ\) has zero curvature if and only if $[X,Y]=0$.
	
	Now we define the Gromoll--Meyer quotient metric $g_0$ as follows. Notice that the quotient map
	$$\pi:\Sp(2)\longrightarrow\Sigma^7$$
	is a smooth surjective submersion. Following \cite{GM}, define the vertical and horizontal spaces by
	\begin{equation}\label{verh}
		\Delta_Q^{\top}
		:=\ker(d\pi)_Q
		=T_Q\bigl(\pi^{-1}(\pi(Q))\bigr),
		\,\,\,\,\,\,\,\,
		\Delta_Q^{\perp}
		:=(\Delta_Q^{\top})^{\perp},\,\,\,\,\,\,\forall Q\in \Sp(2). 
	\end{equation}
	The restriction
	\[
	(d\pi)_Q:\Delta_Q^{\perp}
	\longrightarrow T_{\pi(Q)}\Sigma^7
	\]
	is a linear isomorphism.
	For \(\xi,\eta\in T_{\pi(Q)}\Sigma^7\),
	let \(\widetilde\xi,\widetilde\eta\in \Delta_Q^\perp\) be their unique horizontal lifts at \(Q\). Define the metric $g_0$ on the Gromoll--Meyer sphere $\Sigma^7$ by
	\[(g_{0})_{\pi(Q)}(\xi,\eta)=(\widetilde g_{0})_{Q}(\widetilde\xi,\widetilde\eta).
	\]

	It is clear that
	$\pi:(\Sp(2),\widetilde g_0)
	\longrightarrow(\Sigma^7,g_0)$
	is a Riemannian submersion. For a tangent vector \(Z\in T_Q\Sp(2)\), write
	\[
	Z=Z^{\top}+Z^{\perp},
	\qquad
	Z^{\top}\in\Delta_Q^{\top},
	\qquad
	Z^{\perp}\in\Delta_Q^{\perp},
	\]
	for its orthogonal decomposition into vertical and horizontal components.  For local orthonormal vector fields
	\({\xi},\eta\) on \(\Sigma^7\), let
	\(\widetilde{\xi},\widetilde{\eta}\) denote their horizontal lifts. Then, O'Neill's curvature formula \cite{ONeill} (see
	\cite[p.~403, Eq.~(1)]{GM} as well) yields
	\begin{equation}
		\label{eq:oneill-general}
		\sec_{g_0}(\xi,\eta)\circ\pi
		=
		\sec_{\widetilde g_0}\bigl(\widetilde \xi \wedge\widetilde \eta\bigr)
		+
		\frac34
		\bigl\|\big[\widetilde \xi,\widetilde \eta\big]^{\top}
		\bigr\|_{\widetilde g_0}^{\,2}.
	\end{equation}

	Now let \(p\in\Sigma^7\) and
	\(\sigma\subset T_p\Sigma^7\) be an arbitrary
	two-dimensional subspace. Choose a \(g_0\)-orthonormal
	basis \(\xi,\eta\) of \(\sigma\), and extend it to local
	orthonormal vector fields on \(\Sigma^7\).
	Let \(\mathcal X,\mathcal Y\) denote their horizontal
	lifts. Choose any \(Q\in\pi^{-1}(p)\). 
	Applying O'Neill's formula \eqref{eq:oneill-general},
	we obtain
	\begin{equation*}
		\sec_{g_0,p}(\sigma)=
		\sec_{\widetilde g_0,Q}
		\bigl(\mathcal X_Q\wedge\mathcal Y_Q\bigr)
		+
		\frac34
		\bigl\|[\mathcal X,\mathcal Y]^{\top}_Q
		\bigr\|_{\widetilde g_0}^{\,2}\ge0,
	\end{equation*}
	where the first term is nonnegative by
	\eqref{eq:bi-invariant-curvature}.

	\section{Cheeger deformation and its abstract generalization}
	
	The original Gromoll--Meyer metric has nonnegative sectional curvature,
	but some tangent planes have zero curvature. A natural question is whether one can change the relative lengths
	of different directions while retaining nonnegative curvature.
	
	A compact group of isometries singles out the tangent spaces to its orbits as distinguished directions. Cheeger developed a way to change the metric so that vectors in these directions become shorter while preserving nonnegative sectional curvature \cite{Cheeger}. Its essential idea is to realize the new metric as a quotient of a Riemannian product whose curvature is already known to be nonnegative. O'Neill's formula then controls the curvature of the new metric. We describe the construction on a compact Lie group, following \cite{HT,EK}.

	Let \(K\subset G=\Sp(2)\) be a closed subgroup with Lie algebra
	\(\mathfrak k\).
	For the left \(K\)-action, the orbit through \(Q\) is \(KQ\).
	Differentiating its orbit curves yields that
	\[
	T_Q(KQ)=\{AQ:A\in\mathfrak k\}.
	\]
	For the right \(K\)-action, the corresponding formulas are
	\[
	T_Q(QK)=\{QA:A\in\mathfrak k\}.
	\]
	In our right trivialization, these two orbit tangent spaces
	therefore correspond to
	\[
	T_Q(KQ)\longleftrightarrow\mathfrak k,
	\qquad
	T_Q(QK)\longleftrightarrow\Ad_Q\mathfrak k.
	\]
	
	Define the \(B\)-orthogonal complement of \(\mathfrak k\) by
	\[
	\mathfrak k^\perp
	:=
	\{Z\in\mathfrak g:
	B(Z,A)=0\text{ for every }A\in\mathfrak k\}.
	\]
	Since \(B\) is positive definite, we have the orthogonal
	direct-sum decomposition
	\[
	\mathfrak g=\mathfrak k\oplus\mathfrak k^\perp.
	\]
	Thus every \(X\in\mathfrak g\) can be written uniquely as
	\[
	X=X_{\mathfrak k}+X_{\mathfrak k^\perp},
	\qquad
	X_{\mathfrak k}\in\mathfrak k,
	\qquad
	X_{\mathfrak k^\perp}\in\mathfrak k^\perp.
	\]
	The \(B\)-orthogonal projection onto \(\mathfrak k\) is
	the linear map
	\[
	\Pi_{\mathfrak k}:\mathfrak g\longrightarrow\mathfrak k,
	\qquad
	\Pi_{\mathfrak k}X=X_{\mathfrak k}.
	\]
	Since right translation is an isometry for the original
	bi-invariant metric \(\widetilde g_0\),
	\[
	XQ=X_{\mathfrak k}Q+X_{\mathfrak k^\perp}Q
	\]
	is the orthogonal decomposition into directions tangent
	and perpendicular to the left \(K\)-orbit.

	For \(t>0\), equip \(K\times G\) with the product metric
	\[
	t^{-1}B|_{\mathfrak k}\oplus B.
	\]
	Consider the submersion
	\[
	F:K\times G\longrightarrow G,
	\qquad F(k,Q)=kQ.
	\]
	The fibers of \(F\) are the orbits of the free isometric
	action
	\[
	h\cdot(k,Q)=(kh^{-1},hQ),
	\qquad h\in K.
	\]
	Thus \(F\) defines a quotient metric \(\widetilde g_t\)
	on the same underlying manifold \(G\). Write
	\[
	X=X_{\mathfrak k}+X_{\mathfrak k^\perp}.
	\]
	Computation yields that
	\begin{equation}
		\label{eq:cheeger-scaling}
		(\widetilde g_t)_e(X,X)
		=
		\frac{1}{1+t}\|X_{\mathfrak k}\|_B^2
		+
		\|X_{\mathfrak k^\perp}\|_B^2.
	\end{equation}
	The squared lengths in the orbit directions are multiplied
	by \(1/(1+t)\), while the orthogonal directions are
	unchanged.
	
	Right translations on the \(G\)-factor descend to
	isometries of the quotient. Thus \(\widetilde g_t\) is
	right-invariant and can be written as
	\begin{equation*} (\widetilde g_t)_Q(XQ,YQ)=B(P_tX,Y),
	\end{equation*}
	where
	\begin{equation}\label{Poper}
		P_t:=(I+t\Pi_{\mathfrak k})^{-1}=\frac{1}{1+t}\Pi_{\mathfrak k}+(I-\Pi_{\mathfrak k}).    
	\end{equation}
	The family extends smoothly to \(t=0\), where \(P_0=I\). Then, O'Neill's formula implies $\sec_{\widetilde g_t}\ge0$.
	
	\begin{remark}
		The construction also explains how some initially flat planes can acquire positive curvature. For linearly independent \(x,y\in\mathfrak g\), consider
		the moving plane
		\[\sigma_t=\operatorname{span}\{P_t^{-1}x,P_t^{-1}y\}.
		\]
		The horizontal lifts of its displayed generators at
		\((e,e)\in K\times G\) are $(t x_{\mathfrak k},x)$, $(t y_{\mathfrak k},y)$. If $[x,y]=0$, $[x_{\mathfrak k},y_{\mathfrak k}]\ne0$, then the original plane is flat, but the \(K\)-components of these lifts contribute strictly positive curvature
		in the product. O'Neill's formula therefore gives $\sec_{\widetilde g_t}(\sigma_t)>0$ $(t>0)$. A precise zero-curvature criterion for this deformation
		is given in \cite[Example~3.3]{HT}.    
	\end{remark}
	
	\begin{remark}\label{LG}    
		Eschenburg and Kerin use the nested subgroups
		\[
		H=\{\diag(q,q):q\in\Sp(1)\}
		\subset
		K=\Sp(1)\times\Sp(1)
		\subset
		G=\Sp(2).
		\]
		Let \(\mathfrak h\subset\mathfrak k\subset\mathfrak g\)
		be the corresponding Lie algebras. In our convention the
		two successive contractions have inverse operator
		\begin{equation}
			\label{eq:EK-inverse-operator}
			P_{a,b}^{-1}
			=
			I+a\Pi_{\mathfrak k}+b\Pi_{\mathfrak h},
			\qquad a,b>0.
		\end{equation}
		These are explicit nonnegatively curved metrics obtained by successive Riemannian quotients. The resulting quotient metrics have positive sectional curvature at every point outside an explicitly described
		exceptional set on the Gromoll--Meyer sphere \cite{EK}.

		Wilhelm uses contractions on both sides of \(G\), which can be expressed in our notation by a point-dependent
		operator \(\mathcal P_Q\) satisfying
		\[
		\mathcal P_Q^{-1}
		=
		I+a\Pi_{\mathfrak k}
		+b\Ad_Q\Pi_{\mathfrak k}\Ad_{Q^{-1}}.
		\]
		Wilhelm \cite{Wilh} obtains almost positive curvature for suitable
		parameters. Petersen and Wilhelm \cite{PW,PWprinciples} introduce further changes depending on the point of the manifold.
	\end{remark}

	\subsection{Right-invariant metrics induced by \texorpdfstring{$P$}{dd} operators}
	\label{sec:quotient-metric-framework}
	
	The preceding section constructed Cheeger deformations of \(G=\Sp(2)\) as Riemannian submersion quotients. We now describe, via the operator formalism of \cite{ZK}, a broader class of right-invariant metrics compatible with the Gromoll--Meyer action, from which our candidate metrics will be drawn. Within this class, we then specialize to an inverse-linear family and compute its sectional curvature. Observe that O'Neill's formula no longer guarantees nonnegative sectional curvature in this setting. For a general operator \(P\), the ambient space in Cheeger's construction need not have nonnegative curvature, unlike the product-space ambient spaces arising from Cheeger deformations (see (\ref{Poper} and Remark \ref{LG} for instance).

	Throughout, the underlying Gromoll--Meyer quotient map $\pi:G\rightarrow\Sigma^7$ remains fixed.


	Let $P:\mathfrak {sp}(2)\to\mathfrak {sp}(2)$ be positive definite
	and $B$-self-adjoint.
	Define a metric $\widetilde g_P$ on $\Sp(2)$ by
	\begin{equation}
		\label{eq:general-inertia-metric}
		(\widetilde g_P)_Q(XQ,YQ):=B(PX,Y),
		\qquad X,Y\in\mathfrak {sp}(2).
	\end{equation}
	Recall the Gromoll--Meyer action
	\begin{equation}\label{GMA}
		q\star Q=D(q)QA(q)^{-1},
		\qquad
		D(q)=\diag(q,q),\qquad A(q)=\diag(q,1).    
	\end{equation}
	The action is isometric precisely when 
	\begin{equation}\label{PD}
		P\Ad_{D(q)}=\Ad_{D(q)}P,\,\,\,\,\forall q\in\Sp(1).
	\end{equation}
	For such an operator $P$, the quotient construction defines a smooth metric $g_P$ for which
	$\pi:(G,\widetilde g_P)\rightarrow(\Sigma^7,g_P)$
	is a Riemannian submersion.

	In the coordinates of \eqref{eq:multiplicity-coordinates},
	conjugation by $D(q)$ acts as
	\[
	(r;u_0,u_1,u_2)
	\longmapsto
	(r;qu_0\bar q,qu_1\bar q,qu_2\bar q).
	\]
	It follows that every operator satisfying \eqref{PD} has the form
	\begin{equation}\label{pp}
		P(r;u_0,u_1,u_2)
		=
		\left(
		cr;\,
		\sum_{j=0}^2 C_{0j}u_j,\,
		\sum_{j=0}^2 C_{1j}u_j,\,
		\sum_{j=0}^2 C_{2j}u_j
		\right),
	\end{equation}
	where $c>0$, $C:=(C_{ij})_{0\leq i,j\leq 2}\in M_3(\mathbb R)$ is symmetric and positive definite.
	By (\ref{bp}), we then have
	\[
	B(PZ,Z)
	=
	cr^2+\sum_{i=0}^2C_{ii}|u_i|^2
	+2\sum_{0\le i<j\le2}C_{ij}\,u_i\cdot u_j.
	\]
	
	We now describe the vertical and horizontal spaces of the
	Riemannian submersion $\pi$.
	By (\ref{verh}), the vertical space at $Q$ is the tangent space to the $\Sp(1)$-orbit through $Q$, which is spanned by 
	\begin{equation*}
		\begin{split}
			\left.\frac{d}{dt}\right|_{t=0}(e^{ta}\star Q)
			&=\left.\frac{d}{dt}\right|_{t=0}(\diag(e^{ta},e^{ta})\,Q\,\diag(e^{-ta},1))= \diag(a,a)Q-Q\,\diag(a,0),\,\,    \forall a\in\operatorname{Im}\HH.
		\end{split}    
	\end{equation*}
	Then,
	\begin{equation}
		\label{eq:GM-vertical-space}
		\Delta_Q^\top
		=
		\{J_Q(a)Q:a\in\operatorname{Im}\HH\},
	\end{equation}
	where $J_Q:\operatorname{Im}\HH\rightarrow\mathfrak {sp}(2)$ is the linear map defined by
	\begin{equation}\label{JQ}
		J_Q(a) = \diag(a,a)-Q\,\diag(a,0)\,Q^{-1}.    
	\end{equation}
	
	By definition, the horizontal space $\Delta_{Q}^{\perp}$  consists of all tangent vectors
	$XQ$  satisfying
	\[
	0
	=(\widetilde g_{P})_{Q}\bigl(XQ,J_Q(a)Q\bigr)
	=
	B(PX,J_Q(a)),
	\,\,\,
	\forall a\in\operatorname{Im}\HH.
	\]
	Denote by
	\begin{equation}
		J_Q^*:\mathfrak{sp}(2)\longrightarrow\operatorname{Im}\HH    
	\end{equation}
	the adjoint of $J_Q$ with respect to the inner
	product \eqref{EN} on $\operatorname{Im}\HH$ and the
	inner product \eqref{BN} on $\mathfrak{sp}(2)$, that is, for all $a\in\operatorname{Im}\HH$ and $X\in\mathfrak{sp}(2)$.
	\[
	\langle a,J_Q^*X\rangle_{\mathrm E}
	=
	B(J_Q(a),X)
	\]
	Equivalently,
	\[
	\operatorname{Re}\bigl(\bar a\,J_Q^*X\bigr)
	=
	-\frac12\operatorname{Re}
	\operatorname{tr}_{\HH}\bigl(J_Q(a)X\bigr).
	\]
	With $(e_1,e_2,e_3)=(i,j,k)$, in the multiplicity coordinates
	\[
	X=(r;u_0,u_1,u_2),\qquad
	J_Q(e_\alpha)
	=(r_\alpha;v_{0\alpha},v_{1\alpha},v_{2\alpha}),
	\]
	we have
	\[J_Q^*X= \sum_{\alpha=1}^3 B(J_Q(e_\alpha),X)e_\alpha = \sum_{\alpha=1}^3 \left(r_\alpha r+
	\sum_{j=0}^2
	\operatorname{Re}(\overline{v_{j\alpha}}\,u_j)
	\right)e_\alpha.
	\]
	Consequently, the horizontal condition is
	\begin{equation}
		\label{eq:horizontal-test}
		XQ\in\Delta_Q^\perp
		\quad\Longleftrightarrow\quad
		J_Q^*PX=0\quad\Longleftrightarrow\quad
		c\,r_\alpha r+
		\sum_{i,j=0}^2
		C_{ij}\,v_{i\alpha}\cdot u_j=0,
		\,\,\,\alpha=1,2,3.
	\end{equation}
	
	Let $M_{P,Q}$ be the Gram matrix of the vertical basis
	$J_Q(i)Q,J_Q(j)Q,J_Q(k)Q$ with respect to
	$(\widetilde g_{P})_{Q}$. Equivalently, $M_{P,Q}$ represents $J_Q^*PJ_Q$
	in the basis $(i,j,k)$ of $\operatorname{Im}\HH$. More explicitly,
	\begin{equation}
		\label{eq:deformed-vertical-Gram}
		M_{P,Q}=
		\left(
		c\,r_\alpha r_\beta+
		\sum_{i,j=0}^2
		C_{ij}\,v_{i\alpha}\cdot v_{j\beta}
		\right)_{\alpha,\beta=1}^3.
	\end{equation}
	It is clear that $M_{P,Q}$ is positive definite and invertible.
	
	Motivated by the Cheeger operator \eqref{Poper}, we consider
	an inverse-linear family in the sense of Huizenga--Tapp
	\cite{HT}. Let
	$S:\mathfrak{sp}(2)\to\mathfrak{sp}(2)$ be a
	$B$-self-adjoint operator satisfying
	\begin{equation}\label{SD}
		S\Ad_{D(q)}=\Ad_{D(q)}S,\,\,\,\,\forall q\in\Sp(1).
	\end{equation}
	For $0<\varepsilon\ll1$, the operator
	$I+\varepsilon S$ is positive definite. Then, we define the right-invariant metric $\widetilde g_\varepsilon$ on $\Sp(2)$ induced by
	\begin{equation}\label{eq:framework-inverse-linear}
		P_\varepsilon:=(I+\varepsilon S)^{-1}.
	\end{equation}
	Let $g_\varepsilon$ denote the induced metric on the Gromoll--Meyer sphere.

	According to \eqref{pp}, we can conclude that
	\[
	S(r;u_0,u_1,u_2)
	=
	\left(
	\ell r;\,
	\sum_{j=0}^2T_{0j}u_j,\,
	\sum_{j=0}^2T_{1j}u_j,\,
	\sum_{j=0}^2T_{2j}u_j
	\right),
	\]
	where $\ell\in\mathbb R$ and $T:=(T_{ij})_{0\leq i,j\leq 2}\in M_3(\mathbb R)$ is symmetric such that
	\begin{equation}\label{range}
		1+\varepsilon\ell>0,
		\qquad I_3+\varepsilon T>0.    
	\end{equation}
	The particular operator $S$ will be specified in the construction.
	
	We next express the horizontal spaces of the deformed metrics
	in terms of the horizontal space of the original metric.
	Fix $Q\in \Sp(2)$ and let $\Delta_{\varepsilon,Q}^{\perp}$ denote the orthogonal
	complement of the vertical space $\Delta_Q^\top$ with
	respect to $(\widetilde g_\varepsilon)_Q$. 
	Setting $x=P_\varepsilon X$, we obtain by \eqref{eq:horizontal-test} that
	\[
	J_Q^*x=0,
	\,\,\,\, X=P_\varepsilon^{-1}x=(I+\varepsilon S)x,
	\]
	and hence
	\begin{equation}
		\label{eq:inverse-linear-horizontal-spaces}
		\Delta_{\varepsilon,Q}^{\perp}
		=
		\left\{
		((I+\varepsilon S)x)Q:
		x\in\ker J_Q^*
		\right\}.
	\end{equation}
	Thus the map
	\[
	xQ\longmapsto ((I+\varepsilon S)x)Q
	\]
	is a linear isomorphism from the original horizontal space
	$\Delta_{Q}^{\perp}=\Delta_{0,Q}^{\perp}$ to the deformed horizontal space
	$\Delta_{\varepsilon,Q}^{\perp}$.
	For the curvature calculations below, choose linearly
	independent $x,y\in\ker J_Q^*$ and set
	\begin{equation}
		\label{eq:moving-horizontal-generators}
		X_\varepsilon:=(I+\varepsilon S)x,
		\qquad
		Y_\varepsilon:=(I+\varepsilon S)y.
	\end{equation}
	Then $X_\varepsilon Q$, $Y_\varepsilon Q$ span a
	horizontal two-plane for $\widetilde g_\varepsilon$.
	For each fixed $\varepsilon$, every horizontal two-plane
	arises in this way, for $P_\varepsilon$ is invertible. 

	\subsection{The curvature formulas for the inverse-linear family}
	\label{subsec:ambient-curvature-framework}
	
	We use the curvature convention fixed in (\ref{curvconv}) and (\ref{eq:sectional-curvature}), with a superscript
	on $\mathcal R$ indicating the metric. Throughout this subsection,
	assume that $I+\varepsilon S$ is positive-definite.
	By P\"uttmann's curvature formula \cite{Puttmann} (see \cite{HT} as well), we have
	\begin{lemma}\label{lem:puttmann-curvature}
		For $U,V\in\mathfrak{sp}(2)$, define
		\[
		F(U,V)
		:=
		\frac12\bigl([U,P_{\varepsilon}V]+[V,P_{\varepsilon}U]\bigr).
		\]
		Then, the ambient curvature numerator given by (\ref{eq:sectional-curvature}) satisfies
		\begin{equation}
			\label{eq:puttmann-curvature-general}
			\begin{aligned}
				\mathcal N^{\widetilde g_{\varepsilon}}(U,V)
				={}&
				\frac12B([P_{\varepsilon}U,V]+[U,P_{\varepsilon}V],[U,V])
				-\frac34B(P_{\varepsilon}[U,V],[U,V])\\
				&+B\bigl(F(U,V),P_{\varepsilon}^{-1}F(U,V)\bigr)-B\bigl(F(U,U),P_{\varepsilon}^{-1}F(V,V)\bigr).
			\end{aligned}
		\end{equation}
		In particular, choose, $U=X_\varepsilon$ and $V=Y_\varepsilon$, we have
		\begin{equation}
			\label{eq:invariant-curvature}
			\begin{aligned}
				\mathcal N^{\widetilde g_\varepsilon}
				(X_\varepsilon,Y_\varepsilon)
				={}&
				\frac12B(2z+\varepsilon a_S,Z_\varepsilon)
				-\frac34B(P_\varepsilon Z_\varepsilon,Z_\varepsilon)+\frac{\varepsilon^2}{4}
				B\bigl(h_S,(I+\varepsilon S)h_S\bigr)\\
				&-\varepsilon^2
				B\bigl([Sx,x],(I+\varepsilon S)[Sy,y]\bigr),
			\end{aligned}
		\end{equation}
	\end{lemma}
	
	\begin{proof}
		Let $\nabla$ be the Levi--Civita connection of $\widetilde g_{\varepsilon}$. Throughout the proof, $\langle\cdot,\cdot\rangle$ denotes the inner product with respect to $\widetilde g_{\varepsilon}$ at a fixed point $Q$.
		Since the inner products of right-invariant vector fields are constant,
		we conclude by (\ref{rlie}) and the Koszul formula that
		\[
		\begin{aligned}
			2\langle\nabla_{U^R}V^R,W^R\rangle
			={}&
			-B(P_{\varepsilon}[U,V],W)
			+B(P_{\varepsilon}[V,W],U)
			-B(P_{\varepsilon}[W,U],V)\\
			={}&
			B\bigl(-P_{\varepsilon}[U,V]+[P_{\varepsilon}U,V]-[U,P_{\varepsilon}V],W\bigr).
		\end{aligned}
		\]
		The second equality follows from the invariance of $B$ and the $B$-self-adjointness of $P_\varepsilon$.
		Thus
		\begin{equation}
			\label{eq:puttmann-right-connection}
			\nabla_{U^R}V^R =
			\left(
			-\frac12[U,V]-P_{\varepsilon}^{-1}F(U,V)
			\right)^R.
		\end{equation}
		
		Setting $W:=[U,V]$, we obtain by (\ref{curvconv}) and (\ref{eq:sectional-curvature}) that
		\begin{equation*}
			\begin{split}
				\mathcal N^{\widetilde g_{\varepsilon}}(U,V)
				&=
				\left\langle
				\nabla_{U^R}\nabla_{V^R}V^R,U^R
				\right\rangle
				-
				\left\langle
				\nabla_{V^R}\nabla_{U^R}V^R,U^R
				\right\rangle+
				\left\langle\nabla_{W^R}V^R,U^R\right\rangle\\
				&= -\left\langle
				\nabla_{V^R}V^R,\nabla_{U^R}U^R
				\right\rangle+\left\langle
				\nabla_{U^R}V^R,\nabla_{V^R}U^R
				\right\rangle
				+\left\langle\nabla_{W^R}V^R,U^R\right\rangle.
			\end{split}
		\end{equation*}
		We evaluate these three terms separately.
		By \eqref{eq:puttmann-right-connection},
		\[
		-\left\langle
		\nabla_{V^R}V^R,\nabla_{U^R}U^R
		\right\rangle
		=-B\bigl(F(U,U),P_{\varepsilon}^{-1}F(V,V)\bigr).
		\]
		Since $F(V,U)=F(U,V)$, the same connection formula
		gives
		\[
		\begin{aligned}
			\left\langle
			\nabla_{U^R}V^R,\nabla_{V^R}U^R
			\right\rangle
			&=
			B\left(
			P_{\varepsilon}\left(-\frac12W-P^{-1}F(U,V)\right),
			\frac12W-P_{\varepsilon}^{-1}F(U,V)
			\right)\\
			&= -\frac14B(P_{\varepsilon}[U,V],[U,V]) +B\bigl(F(U,V),P_{\varepsilon}^{-1}F(U,V)\bigr).
		\end{aligned}
		\]
		Finally, applying the Koszul formula once more yields
		\[
		\begin{aligned}
			2\left\langle\nabla_{W^R}V^R,U^R\right\rangle
			={}&
			-B(P_{\varepsilon}[W,V],U)
			+B(P_{\varepsilon}[V,U],W)
			-B(P_{\varepsilon}[U,W],V)\\
			={}&
			B([P_{\varepsilon}U,V]+[U,P_{\varepsilon}V],[U,V])-B(P_{\varepsilon}[U,V],[U,V]).
		\end{aligned}
		\]
		Adding these three expressions proves
		\eqref{eq:puttmann-curvature-general}.
	\end{proof}

	Take $x,y\in\mathfrak{sp}(2)$ and, as in \cite[Section~4]{HT}, set
	\begin{equation}
		\label{eq:inverse-linear-bracket-data}
		\begin{split}
			&z:=[x,y],\qquad
			a_S:=[Sx,y]+[x,Sy],\qquad
			b_S:=[Sx,Sy], \qquad h_S:=[Sx,y]-[x,Sy].
		\end{split} 
	\end{equation}
	Let \(X_\varepsilon=(I+\varepsilon S)x\), \(Y_\varepsilon=(I+\varepsilon S)y\)
	be given by
	\eqref{eq:moving-horizontal-generators}, and set
	\[
	Z_\varepsilon
	:=[X_\varepsilon,Y_\varepsilon]=z+\varepsilon a_S+\varepsilon^2b_S.
	\]
	Since $P_\varepsilon X_\varepsilon=x$, $P_\varepsilon Y_\varepsilon=y$, we have
	\begin{equation*}
		\begin{split}
			&[P_\varepsilon X_\varepsilon,Y_\varepsilon]   +[X_\varepsilon,P_\varepsilon Y_\varepsilon]=[x,y+\varepsilon Sy]+[x+\varepsilon Sx,y]
			=2z+\varepsilon a_S,\\
			&F(X_\varepsilon,Y_\varepsilon)=\frac12\bigl([x+\varepsilon Sx,y]-[x,y+\varepsilon Sy]\bigr)=\frac{\varepsilon}{2}\bigl([Sx,y]-[x,Sy]\bigr)
			=\frac{\varepsilon}{2}h_S,\\
			&F(X_\varepsilon,X_\varepsilon)=[X_\varepsilon,P_\varepsilon X_\varepsilon] =\varepsilon [Sx,x],\,\,F(Y_\varepsilon,Y_\varepsilon)=[Y_\varepsilon,P_\varepsilon Y_\varepsilon]
			=\varepsilon [Sy,y].   
		\end{split}    
	\end{equation*}
	Substituting these identities into
	\eqref{eq:puttmann-curvature-general}, we derive \eqref{eq:invariant-curvature}.

	We now pass to the quotient. Fix $Q\in G=\Sp(2)$ and
	linearly independent $x,y\in\ker J_Q^*$. Keep $\varepsilon$
	fixed throughout the following calculation, with
	$I+\varepsilon S$ positive definite, and let
	$X_\varepsilon,Y_\varepsilon$ be as in
	\eqref{eq:moving-horizontal-generators}. By
	\eqref{eq:inverse-linear-horizontal-spaces}, the vectors
	$X_\varepsilon Q,Y_\varepsilon Q$ are horizontal for
	$\widetilde g_\varepsilon$.
	Write 
	\[
	\xi_\varepsilon:=(d\pi)_Q(X_\varepsilon Q),
	\qquad
	\eta_\varepsilon:=(d\pi)_Q(Y_\varepsilon Q).
	\]
	The restriction
	\[
	(d\pi)_Q:
	\Delta_{\varepsilon,Q}^{\perp}
	\longrightarrow T_{\pi(Q)}\Sigma^7
	\]
	is a linear isometry. In particular,
	$\xi_\varepsilon,\eta_\varepsilon$ are linearly independent
	and have the same Gram matrix as their horizontal lifts.
	
	Extend $\xi_\varepsilon,\eta_\varepsilon$ to smooth local
	vector fields on $\Sigma^7$, and denote their horizontal lifts
	by $\mathcal X,\mathcal Y$. Thus
	\[
	\mathcal X_Q=X_\varepsilon Q,
	\qquad
	\mathcal Y_Q=Y_\varepsilon Q.
	\]
	In the following, the superscript $\top$ denotes
	orthogonal projection onto the vertical space
	$\Delta_Q^\top$ with respect to
	$(\widetilde g_\varepsilon)_Q$. 
	Using the curvature-numerator notation of
	\eqref{eq:sectional-curvature}, O'Neill's formula
	\eqref{eq:oneill-general} for the deformed submersion becomes
	\begin{equation}
		\label{eq:oneill}
		\mathcal N^{g_\varepsilon}
		(\xi_\varepsilon,\eta_\varepsilon)
		=
		\mathcal N^{\widetilde g_\varepsilon}
		(X_\varepsilon,Y_\varepsilon)
		+\frac34
		\bigl\|[\mathcal X,\mathcal Y]^\top_Q
		\bigr\|_{\widetilde g_\varepsilon}^{\,2}.
	\end{equation}
	No orthonormality assumption is needed in
	\eqref{eq:oneill}. Indeed, applying the orthonormal formula after Gram--Schmidt and multiplying by the common Gram determinant gives this identity.

	To evaluate the last term in \eqref{eq:oneill}, we use the biquotient framework of \cite{ZK} and give the calculation explicitly in our right-invariant convention.
	Identify the Gram matrix in \eqref{eq:deformed-vertical-Gram} with its operator on
	$\operatorname{Im}\HH$ in the orthonormal basis $(i,j,k)$, and write
	\[
	M_{\varepsilon,Q}
	:=
	M_{P_\varepsilon,Q}
	=
	J_Q^*P_\varepsilon J_Q.
	\]

	Define an $\operatorname{Im}\HH$-valued one-form $\lambda_\varepsilon$ on $\Sp(2)$ by
	\[ (\lambda_\varepsilon)_{Q'}(V) := J_{Q'}^*P_\varepsilon\bigl(V(Q')^{-1}\bigr), \qquad
	Q'\in \Sp(2),\,\, V\in T_{Q'}\Sp(2).
	\]
	By the definition of $J_{Q'}^*$, for each $a\in\operatorname{Im}\HH$,
	\[ \langle a,(\lambda_\varepsilon)_{Q'}(V)\rangle_{\mathrm E}=
	B\bigl(J_{Q'}(a),
	P_\varepsilon(V(Q')^{-1})\bigr)=
	(\widetilde g_\varepsilon)_{Q'}
	\bigl(J_{Q'}(a)Q',V\bigr).
	\]
	In particular, by \eqref{eq:GM-vertical-space},
	\[
	\ker(\lambda_\varepsilon)_{Q'}
	=
	\Delta_{\varepsilon,Q'}^\perp.
	\]

	Set
	\[
	\beta_{\varepsilon,Q}(x,y)
	:=
	(d\lambda_\varepsilon)_Q
	(X_\varepsilon Q,Y_\varepsilon Q)=(d\lambda_\varepsilon^1)_Q(X_\varepsilon Q,Y_\varepsilon Q)\,i +(d\lambda_\varepsilon^2)_Q(X_\varepsilon Q,Y_\varepsilon Q)\,j+(d\lambda_\varepsilon^3)_Q(X_\varepsilon Q,Y_\varepsilon Q)\,k.
	\]
	For any local vector fields
	$\mathcal U,\mathcal V$,  Cartan's formula gives
	\begin{equation}\label{Cart}
		d\lambda_\varepsilon(\mathcal U,\mathcal V)
		= \mathcal U\bigl(\lambda_\varepsilon(\mathcal V)\bigr)
		-
		\mathcal V\bigl(\lambda_\varepsilon(\mathcal U)\bigr)-
		\lambda_\varepsilon([\mathcal U,\mathcal V]).    
	\end{equation}
	For the horizontal lifts $\mathcal X,\mathcal Y$, $\lambda_\varepsilon(\mathcal X) = \lambda_\varepsilon(\mathcal Y) = 0$. Consequently,
	\[ \beta_{\varepsilon,Q}(x,y)=
	-(\lambda_\varepsilon)_Q
	([\mathcal X,\mathcal Y]_Q)=
	-(\lambda_\varepsilon)_Q
	([\mathcal X,\mathcal Y]^\top_Q).
	\]
	The second equality uses the fact that
	$\lambda_\varepsilon$ annihilates the horizontal component.
	
	Write
	\[
	[\mathcal X,\mathcal Y]^\top_Q=J_Q(c)Q
	\]
	for the unique $c\in\operatorname{Im}\HH$. Since
	\[
	(\lambda_\varepsilon)_Q(J_Q(c)Q)
	=
	J_Q^*P_\varepsilon J_Q(c)
	=
	M_{\varepsilon,Q}c,
	\]
	we obtain
	\[
	[\mathcal X,\mathcal Y]^\top_Q
	=
	-J_Q\bigl(
	M_{\varepsilon,Q}^{-1}
	\beta_{\varepsilon,Q}(x,y)
	\bigr)Q.
	\]
	Moreover,
	\[ \|J_Q(c)Q\|_{\widetilde g_\varepsilon}^2=
	B(P_\varepsilon J_Q(c),J_Q(c))=
	\langle c,M_{\varepsilon,Q}c\rangle_{\mathrm E}.
	\]
	Substituting
	$c=-M_{\varepsilon,Q}^{-1}\beta_{\varepsilon,Q}(x,y)$
	proves
	\begin{equation}
		\label{eq:oneill-matrix-term}
		\bigl\|[\mathcal X,\mathcal Y]^\top_Q
		\bigr\|_{\widetilde g_\varepsilon}^{\,2}
		=
		\left\langle
		\beta_{\varepsilon,Q}(x,y),
		M_{\varepsilon,Q}^{-1}\beta_{\varepsilon,Q}(x,y)
		\right\rangle_{\mathrm E}.
	\end{equation}
	
	It remains to evaluate $\beta_{\varepsilon,Q}(x,y)$. For $a\in \operatorname{Im}\HH$, define $D(a)=\diag(a,a)$.
	In multiplicity coordinates, 
	\[
	B(D(a),(r;u_0,u_1,u_2))
	=
	\langle a,u_0\rangle_{\mathrm E}.
	\]
	Therefore its adjoint with respect to
	$\langle\cdot,\cdot\rangle_{\mathrm E}$ and $B$ satisfies
	\[
	D^*(r;u_0,u_1,u_2)=u_0.
	\]
	For fixed $U\in\mathfrak{sp}(2)$ and
	$a\in\operatorname{Im}\HH$, (\ref{JQ}) yields
	\[
	J_{e^{sU}Q}(a)
	=
	D(a)
	-
	e^{sU}\bigl(Q\diag(a,0)Q^{-1}\bigr)e^{-sU}.
	\]
	Differentiating at $s=0$, we obtain
	\[ \left.\frac{d}{ds}\right|_{s=0}
	J_{e^{sU}Q}(a)=
	-[U,Q\diag(a,0)Q^{-1}]=
	[U,J_Q(a)-D(a)].
	\]
	
	Although the identity
	\[
	d\lambda_\varepsilon(\mathcal X,\mathcal Y)
	=
	-\lambda_\varepsilon([\mathcal X,\mathcal Y])
	\]
	uses horizontal extensions, $d\lambda_\varepsilon$ itself is a two-form. Its value at $Q$ therefore depends only on the two tangent vectors at $Q$. We may consequently compute
	it using the right-invariant fields
	$X_\varepsilon^R,Y_\varepsilon^R$. 
	To compute componentwise, fix $a\in\operatorname{Im}\HH$
	and pair $\lambda_\varepsilon$ with $a$. Fixing $X_\varepsilon,Y_\varepsilon,x,y$ as Lie-algebra elements and varing $Q'$, we have
	\[
	\begin{aligned}
		\bigl\langle a,
		\lambda_\varepsilon(Y_\varepsilon^R)
		\bigr\rangle_{\mathrm E}(Q')
		&=
		B(J_{Q'}(a),P_\varepsilon Y_\varepsilon)
		=
		B(J_{Q'}(a),y),\\
		\bigl\langle a,
		\lambda_\varepsilon(X_\varepsilon^R)
		\bigr\rangle_{\mathrm E}(Q')
		&=
		B(J_{Q'}(a),P_\varepsilon X_\varepsilon)
		=
		B(J_{Q'}(a),x).
	\end{aligned}
	\]
	Differentiating
	along the corresponding flows therefore yields
	\[
	\begin{aligned}
		X_\varepsilon^R
		\bigl\langle a,
		\lambda_\varepsilon(Y_\varepsilon^R)
		\bigr\rangle_{\mathrm E}\big|_Q
		&=
		B([X_\varepsilon,J_Q(a)-D(a)],y),\\
		Y_\varepsilon^R
		\bigl\langle a,
		\lambda_\varepsilon(X_\varepsilon^R)
		\bigr\rangle_{\mathrm E}\big|_Q
		&=
		B([Y_\varepsilon,J_Q(a)-D(a)],x).
	\end{aligned}
	\]
	Recall also that
	$Z_\varepsilon=[X_\varepsilon,Y_\varepsilon]$.
	By our right-invariant bracket convention \eqref{rlie},
	\[
	[X_\varepsilon^R,Y_\varepsilon^R]_Q
	=
	-Z_\varepsilon Q.
	\]
	The minus sign in the exterior-derivative formula thus gives
	\[
	-\bigl\langle a,
	(\lambda_\varepsilon)_Q
	([X_\varepsilon^R,Y_\varepsilon^R]_Q)
	\bigr\rangle_{\mathrm E}
	=
	B(J_Q(a),P_\varepsilon Z_\varepsilon).
	\]
	Combining these three terms, we find by (\ref{Cart}) that
	\[
	\begin{aligned}
		\langle a,\beta_{\varepsilon,Q}(x,y)\rangle_{\mathrm E}
		&=B([X_\varepsilon,J_Q(a)-D(a)],y)-
		B([Y_\varepsilon,J_Q(a)-D(a)],x)+
		B(J_Q(a),P_\varepsilon Z_\varepsilon)\\
		&= B(D(a),2z+\varepsilon a_S)+
		B\bigl(
		J_Q(a),
		P_\varepsilon Z_\varepsilon-2z-\varepsilon a_S
		\bigr).\\
	\end{aligned}
	\]
	Since this holds for every $a\in\operatorname{Im}\HH$,
	the definitions of $D^*$ and $J_Q^*$ give the exact identity
	\begin{equation}
		\label{eq:beta-exact-ped}
		\beta_{\varepsilon,Q}(x,y)
		= D^*(2z+\varepsilon a_S)+
		J_Q^*\bigl(
		P_\varepsilon Z_\varepsilon-2z-\varepsilon a_S
		\bigr).
	\end{equation}
	Finally, substituting \eqref{eq:oneill-matrix-term} into
	\eqref{eq:oneill} yields 
	\begin{equation}
		\label{eq:quotient-curvature-ped}
		\mathcal N^{g_\varepsilon}
		(\xi_\varepsilon,\eta_\varepsilon)
		=\mathcal N^{\widetilde g_\varepsilon}
		(X_\varepsilon,Y_\varepsilon)+
		\frac34
		\left\langle
		\beta_{\varepsilon,Q}(x,y),
		M_{\varepsilon,Q}^{-1}\beta_{\varepsilon,Q}(x,y)
		\right\rangle_{\mathrm E}.
	\end{equation}
	
	\begin{remark}[Deforming the original zero-curvature planes]
		Fix $Q\in G$ and linearly independent $x,y\in\ker J_Q^*$
		with $[x,y]=0$. Let $X_\varepsilon,Y_\varepsilon$ and
		$\xi_\varepsilon,\eta_\varepsilon$ be the corresponding
		moving horizontal generators and their projections.
		Using the notation of \eqref{eq:inverse-linear-bracket-data},
		set
		\[
		w_S:=b_S-Sa_S,
		\qquad
		\zeta_S(x,y):=D^*a_S,
		\]
		and define
		\begin{equation}
			\label{eq:HT-cubic-coefficient}
			\begin{aligned}
				\delta_S(x,y):={}&
				-B(a_S,b_S)
				+\frac34 B(Sa_S,a_S)
				+\frac14 B(Sh_S,h_S)\\
				&-B\bigl(S[Sx,x],[Sy,y]\bigr).
			\end{aligned}
		\end{equation}
		For $I+\varepsilon S>0$, Huizenga--Tapp's exact expansion
		\cite[Proposition~3.1]{HT}, with $\Psi=-S$, gives
		\begin{equation}
			\label{eq:HT-commuting-expansion}
			\mathcal N^{\widetilde g_\varepsilon}
			(X_\varepsilon,Y_\varepsilon)
			=
			\varepsilon^3\delta_S(x,y)
			-\frac34\varepsilon^4 B(P_\varepsilon w_S,w_S).
		\end{equation}
		In particular, the ambient terms of orders zero, one,
		and two vanish.
		
		Since
		\[
		P_\varepsilon Z_\varepsilon
		=\varepsilon a_S+\varepsilon^2P_\varepsilon w_S,
		\]
		the exact quotient formula \eqref{eq:beta-exact-ped} becomes
		\begin{equation}
			\label{eq:HT-quotient-beta}
			\beta_{\varepsilon,Q}(x,y)
			=
			\varepsilon\zeta_S(x,y)
			+\varepsilon^2J_Q^*P_\varepsilon w_S.
		\end{equation}
		Thus the quotient curvature can already have a positive
		quadratic term. More precisely, if $\zeta_S(x,y)\ne0$,
		then \eqref{eq:quotient-curvature-ped} yields
		\[
		\mathcal N^{g_\varepsilon}
		(\xi_\varepsilon,\eta_\varepsilon)
		=
		\frac34\varepsilon^2
		\left\langle
		\zeta_S(x,y),(J_Q^*J_Q)^{-1}\zeta_S(x,y)
		\right\rangle_{\mathrm E}
		+O(\varepsilon^3).
		\]
		Its quadratic coefficient is strictly positive, so this
		fixed moving plane has positive curvature for sufficiently
		small $\varepsilon>0$.
		
		If $\zeta_S(x,y)=0$, let
		\[
		w^{\mathrm h}_{S,\varepsilon,Q}
		:=
		w_S-J_QM_{\varepsilon,Q}^{-1}J_Q^*P_\varepsilon w_S.
		\]
		Then
		$w^{\mathrm h}_{S,\varepsilon,Q}Q=(w_SQ)^\perp$,
		where the projection is taken with respect to
		$\widetilde g_\varepsilon$.
		The O'Neill term cancels the vertical part of the
		fourth-order ambient remainder, giving the exact identity
		\begin{equation}
			\label{eq:critical-horizontal-remainder}
			\begin{aligned}
				\mathcal N^{g_\varepsilon}
				(\xi_\varepsilon,\eta_\varepsilon)
				={}&\varepsilon^3\delta_S(x,y)\\
				&-\frac34\varepsilon^4
				B\bigl(P_\varepsilon w^{\mathrm h}_{S,\varepsilon,Q},
				w^{\mathrm h}_{S,\varepsilon,Q}\bigr).
			\end{aligned}
		\end{equation}
		This explains why the subsequent analysis must examine
		$\delta_S$ on the critical pairs satisfying
		\[
		J_Q^*x=J_Q^*y=0,
		\qquad [x,y]=0,
		\qquad D^*a_S=0.
		\]
		Positivity along each fixed original zero-curvature plane
		does not by itself establish positive curvature for one
		common interval of parameters. That conclusion also requires
		uniform control of nearby planes, including noncommuting
		pairs, using \eqref{eq:invariant-curvature} and
		\eqref{eq:quotient-curvature-ped}.
	\end{remark}
	
	\section{Construction of the candidate metric}
	\label{sec}

	Use the decomposition
	\[
	\mathfrak h=U_0(\operatorname{Im}\HH),\qquad
	\mathfrak k\ominus\mathfrak h=U_1(\operatorname{Im}\HH),\qquad
	\mathfrak k^{\perp_B}=\R m\oplus U_2(\operatorname{Im}\HH).
	\]
	The constant positive operators compatible with the quotient symmetry are
	\[
	P=c\,\operatorname{Id}_{\R m}\oplus(Q\otimes I_3),
	\qquad c>0,\quad Q=Q^t>0.
	\]

	For a symmetric $3\times3$ matrix $T$ and a real number $\ell$, put
	\begin{equation}\label{eq:general-S-choice}
		S_{\ell,T}=\ell\,\operatorname{Id}_{\R m}\oplus(T\otimes I_3),
	\end{equation}
	and set
	\[
	h_t(X,Y)=B((I+tS_{\ell,T})^{-1}X,Y),
	\qquad P_t=(I+tS_{\ell,T})^{-1}.
	\]
	The horizontal spaces are
	\[
	\mathcal H_{t,g}=(I+tS_{\ell,T})\ker J_g^*.
	\]
	
	Choose
	\[
	T=\begin{pmatrix}0&1&0\\1&0&-1\\0&-1&0\end{pmatrix},
	\qquad \ell=3,
	\]
	and write $S_3=S_{3,T}$.  Thus
	\begin{equation}\label{eq:candidate-S}
		\boxed{S_3(r;u_0,u_1,u_2)=(3r;u_1,u_0-u_2,-u_1).}
	\end{equation}
	Equivalently, for imaginary $u,v,w$ and real $r$,
	\[
	S_3\begin{pmatrix}u&r+v\\-r+v&w\end{pmatrix}
	=\begin{pmatrix}u-v&3r+(w-u)/2\\-3r+(w-u)/2&v-w\end{pmatrix}.
	\]
	Moreover,
	\[
	T^2=\begin{pmatrix}1&0&-1\\0&2&0\\-1&0&1\end{pmatrix},
	\qquad T^3=2T,
	\qquad \det(\lambda I-T)=\lambda(\lambda^2-2).
	\]
	
	For $g\in G$ and tangent vectors $V,W\in T_gG$, we have $Vg^{-1}, W g^{-1}\in  \mathfrak{sp}(2)$. Define
	\begin{equation}\label{eq:candidate-metric}
		\boxed{\begin{gathered}
				P_\varepsilon=(I+\varepsilon S_3)^{-1},\qquad
				\widetilde g_{\varepsilon,g}(V,W)
				=B(P_\varepsilon (Vg^{-1}),Wg^{-1}),\\
				g_\varepsilon=\widetilde g_\varepsilon\big/\Sp(1).
		\end{gathered}}
	\end{equation}
	
	The metric is right-$\Sp(2)$-invariant and invariant under $g\mapsto diag(q,q)g$. Hence the Gromoll–Meyer action is isometric and the metric descends to
	$\Sigma^7_{\rm GM}$.  The inertia operator is
	\[
	P_\varepsilon m=(1+3\varepsilon)^{-1}m,
	\]
	\[
	P_\varepsilon|_{U_0\oplus U_1\oplus U_2}
	=\left(I-\frac{\varepsilon T}{1-2\varepsilon^2}
	+\frac{\varepsilon^2T^2}{1-2\varepsilon^2}\right)\otimes I_3.
	\]
	The eigenvalues of $S_3$ are $3,0,\sqrt2,-\sqrt2$ with multiplicities
	$1,3,3,3$.  Hence $P_\varepsilon$ is positive for
	$0<\varepsilon<1/2$.
	
	\begin{theorem}[The fixed $S_3$ conclusion]\label{thm:main}
		For the operator \eqref{eq:candidate-S}, there exists
		$\varepsilon_0>0$ such that the quotient metric
		\eqref{eq:candidate-metric} has $\sec(g_\varepsilon)>0$ whenever
		$0<\varepsilon\le\varepsilon_0$.
	\end{theorem}
	
	The proof occupies the remainder of the main argument.  It combines the
	raw critical-set inequality of Proposition~\ref{s3:hand:positive} with
	the exact curvature and uniform estimates in
	Section~\ref{s3:sec:analytic}; the final assembly is
	Theorem~\ref{s3:thm:positive-final}.  The common parameter is specified
	by \eqref{s3:eq:epsilon0}, using a compact minimum whose strict
	positivity is proved.  No machine-generated rational certificate is a
	premise of this proof, and no claim about another value of $\ell$ is
	needed.  The earlier construction claim in \cite{PW} remains direct
	prior work; no priority claim is made here.
	
	\section{Reduction to a critical-set inequality}
	\label{sec:roadmap}\label{sec:criterion}
	\label{s3:sec:critical-system}
	
	We prove Theorem~\ref{thm:main} for the candidate metric. The argument has two independent
	parts. An analytic estimate controls the curvature uniformly over all
	planes. An algebraic inequality supplies the sign of its cubic term on
	the set where the lower-order positive terms vanish. We state both results
	after fixing the notation, and assemble them in
	Section~\ref{rw:completion}.
	
	\subsection{A fixed space of horizontal frames}
	
	Write \(G=\Sp(2)\), \(\mathfrak g=\mathfrak{sp}(2)\), and
	\[
	B(U,V)=-\frac12\Rea\tr_{\HH}(UV),\qquad |U|^2=B(U,U).
	\]
	Thus \(B\) is the fixed Euclidean inner product on \(\mathfrak g\).
	For imaginary quaternions, the dot and norm denote the ordinary Euclidean
	products on \(\im\HH=\R^3\), with
	\(uv=-u\cdot v+u\times v\). We use the coordinates
	\begin{equation}\label{rw:coordinates}
		(r;u_0,u_1,u_2)=
		\begin{pmatrix}
			u_0+u_1&r+u_2\\-r+u_2&u_0-u_1
		\end{pmatrix},
		\qquad
		B((r;u_i),(s;v_i))=rs+\sum_{i=0}^2u_i\cdot v_i.
	\end{equation}
	The indices label the three imaginary-coordinate blocks; they are not
	components of a vector in \(\R^3\).
	
	Fix
	\[
	S=S_3,\qquad
	S(r;u_0,u_1,u_2)=(3r;u_1,u_0-u_2,-u_1),\qquad
	Q_\varepsilon=I+\varepsilon S,\qquad P_\varepsilon=Q_\varepsilon^{-1}.
	\]
	The operators are self-adjoint for \(B\), and \(\|S\|=3\).
	We restrict initially to \(0\le\varepsilon\le\varepsilon_*\), where
	\(\varepsilon_*>0\) is small enough that
	\begin{equation}\label{rw:metric-comparison}
		\frac12I\le Q_\varepsilon\le\frac32I,
		\qquad \frac23I\le P_\varepsilon\le2I.
	\end{equation}
	For example, \(\varepsilon_*=1/6\) has this property.
	A tangent vector at \(g\) is written \(Ug\), with \(U\in\mathfrak g\), and
	\[
	\widetilde g_{\varepsilon,g}(Ug,Vg)=B(P_\varepsilon U,V).
	\]
	
	The quotient map \(\pi:G\to\Sigma\) is defined by the free action
	\(k\star g=\diag(k,k)g\diag(k,1)^{-1}\), where \(k\in\Sp(1)\).
	For \(\xi\in\im\HH\), put
	\begin{equation}\label{rw:orbit-maps}
		D\xi=\diag(\xi,\xi),\qquad
		E_g\xi=g\diag(\xi,0)g^*,\qquad J_g=D-E_g.
	\end{equation}
	A star on a quaternionic matrix means conjugate transpose. On a map from
	\(\im\HH\) to \(\mathfrak g\), it means the adjoint for the two fixed
	real inner products; for instance,
	\(\xi\cdot J_g^*U=B(J_g\xi,U)\).
	The useful formulas are
	\begin{equation}\label{s3:detail:adjoints}
		D^*(r;u_0,u_1,u_2)=u_0,\qquad
		E_g^*U=\frac12(g^*Ug)_{11}.
	\end{equation}
	They follow by taking the real trace in the definition of \(B\).
	Differentiating the orbit action gives the vertical vectors
	\((J_g\xi)g\). Hence
	\begin{equation}\label{rw:horizontal}
		Ug\text{ is horizontal for }\widetilde g_\varepsilon
		\quad\Longleftrightarrow\quad J_g^*P_\varepsilon U=0.
	\end{equation}
	
	Define the parameter-independent space
	\begin{equation}\label{rw:incidence}
		\mathcal I=\left\{(g,x,y):
		\begin{array}{l}
			g\in G,\quad x,y\in\mathfrak g,\quad J_g^*x=J_g^*y=0,\\
			|x|=|y|=1,\quad B(x,y)=0
		\end{array}\right\}.
	\end{equation}
	It is a closed subset of \(G\) times the Euclidean orthonormal
	two-frame space of \(\mathfrak g\), and is therefore compact.
	For a triple in \(\mathcal I\), set
	\begin{equation}\label{rw:moving-frame}
		X=Q_\varepsilon x,\qquad Y=Q_\varepsilon y.
	\end{equation}
	The vectors \(Xg,Yg\) are horizontal, since \(P_\varepsilon X=x\) and
	\(P_\varepsilon Y=y\).
	
	Conversely, let \(\Pi\subset\mathfrak g\) be the right-translated
	Lie algebra plane of any horizontal two-plane at \(g\).
	Equation~\eqref{rw:horizontal} gives
	\(P_\varepsilon\Pi\subset\ker J_g^*\).
	Choose a \(B\)-orthonormal basis \(x,y\) of \(P_\varepsilon\Pi\).
	Then \((g,x,y)\in\mathcal I\) and \(Q_\varepsilon x,Q_\varepsilon y\)
	span \(\Pi\). Thus one estimate on \(\mathcal I\) controls every
	horizontal plane of every metric in the parameter interval.
	
	Let \(N_\varepsilon(g,x,y)\) denote the quotient curvature numerator
	on \(d\pi_g(Xg),d\pi_g(Yg)\). Our curvature convention is
	\[
	R(U,V)W=\nabla_U\nabla_VW-\nabla_V\nabla_UW-\nabla_{[U,V]}W,
	\qquad N(U,V)=\langle R(U,V)V,U\rangle.
	\]
	The numerator has the sign of sectional curvature because the Gram
	determinant of an independent pair is positive.
	
	\subsection{The two estimates}
	
	For any \(x,y\in\mathfrak g\), define
	\begin{equation}\label{s3:eq:coeffs}
		\begin{gathered}
			c=[x,y],\qquad a=[Sx,y]+[x,Sy],\qquad b=[Sx,Sy],\\
			h=[Sx,y]-[x,Sy],\qquad f_x=[Sx,x],\qquad f_y=[Sy,y],\\
			\zeta=D^*a.
		\end{gathered}
	\end{equation}
	Thus \(c,a,b,h,f_x,f_y\) are Lie algebra vectors, whereas
	\(\zeta\in\im\HH\). The scalar used in the proof is
	\begin{equation}\label{s3:eq:delta}
		\delta_3(x,y)=-B(a,b)+\frac34B(Sa,a)
		+\frac14B(Sh,h)-B(Sf_x,f_y).
	\end{equation}
	On a commuting pair, it is the cubic coefficient of the moving-plane
	ambient curvature. Away from commuting pairs, the complete cubic
	coefficient has additional terms, displayed in
	\eqref{s3:analytic:ambient}.
	
	\begin{proposition}\label{rw:uniform-proposition}
		There are constants \(a_0>0\), \(C_0>0\), and
		\(0<\varepsilon_1\le\varepsilon_*\), independent of
		\((g,x,y)\in\mathcal I\), such that
		\begin{equation}\label{s3:eq:uniform}
			N_\varepsilon(g,x,y)\ge
			a_0\bigl(|c|^2+\varepsilon^2|\zeta|^2\bigr)
			+\varepsilon^3\delta_3-C_0\varepsilon^4
			\qquad(0<\varepsilon\le\varepsilon_1).
		\end{equation}
		The constant \(C_0\) may also be chosen so that
		\(|\delta_3|\le C_0\) on \(\mathcal I\).
	\end{proposition}
	
	The two nonnegative terms in this estimate single out the set
	\begin{equation}\label{s3:eq:raw}
		\mathcal Z=\{(g,x,y)\in\mathcal I:c=0,\ \zeta=0\}.
	\end{equation}
	
	\begin{theorem}\label{s3:thm:critical}
		For every \((g,x,y)\in\mathcal Z\),
		\[
		\delta_3(x,y)\ge\frac{\sqrt5-1}{4}>\frac14.
		\]
	\end{theorem}
	
	Proposition~\ref{rw:uniform-proposition} is proved in
	Section~\ref{s3:sec:analytic}. Theorem~\ref{s3:thm:critical} is proved in
	Sections~\ref{s3:detail:spectral}--\ref{s3:hand:critical}.
	The distinction between the undeformed zero set and \(\mathcal Z\)
	is important. The baseline quotient numerator vanishes exactly when
	\(c=0\). On a fixed commuting pair with \(\zeta\ne0\), the O'Neill
	term supplies a positive quadratic coefficient. Only when
	\(c=\zeta=0\) does the argument require positivity of the cubic
	coefficient. These assertions follow from the exact formulas below.
	
	The algebraic proof uses horizontality to obtain
	\(1/4\le |D^*x|^2+|D^*y|^2\le1/2\), then proves the cubic inequality
	on the larger set of orthonormal pairs satisfying this bound together
	with \(c=\zeta=0\). The analytic proof treats all of \(\mathcal I\)
	without dividing by \(|c|\) or \(|\zeta|\).
	Finally, compactness separates the region where \(\delta_3\) is small
	from \(\mathcal Z\); on its complement the cubic term absorbs the
	fourth-order error. This yields one parameter interval for all planes.
	
	\section{Curvature and uniform control}
	\label{s3:sec:analytic}
	
	Throughout this section \(S=S_3\), \(Q=I+\varepsilon S\), and
	\(P=Q^{-1}\). The pair \(x,y\) belongs to the fixed space
	\(\mathcal I\) in \eqref{rw:incidence}, and \(X=Qx,Y=Qy\).
	All unmarked norms and adjoints use \(B\) on \(\mathfrak g\) and the
	Euclidean product on \(\im\HH\). We use the bracket expressions
	\(c,a,b,h,f_x,f_y,\zeta,\delta_3\) from
	\eqref{s3:eq:coeffs}--\eqref{s3:eq:delta}.
	We first derive exact formulas, and then estimate their remainders.
	
	\subsection{The invariant-metric identity}
	
	For a positive self-adjoint operator \(P\) on \(\mathfrak g\), write
	\(\langle U,V\rangle_P=B(PU,V)\), and define the symmetric
	\(\mathfrak g\)-valued bilinear map
	\[
	F_P(U,V)=\frac12\bigl([U,PV]-[PU,V]\bigr).
	\]
	The reference inner product satisfies
	\begin{equation}\label{s3:analytic:adinvariance}
		B([U,V],W)=B(U,[V,W]).
	\end{equation}
	This follows from cyclicity of the real quaternionic trace.
	
	The invariant-metric curvature identity in our conventions is
	\begin{equation}\label{s3:analytic:invariantcurvature}
		\begin{aligned}
			N_P(U,V)={}&
			\frac12B([PU,V]+[U,PV],[U,V])
			-\frac34B(P[U,V],[U,V])\\
			&+B(F_P(U,V),P^{-1}F_P(U,V))
			-B(F_P(U,U),P^{-1}F_P(V,V)).
		\end{aligned}
	\end{equation}
	Here \(N_P(U,V)\) is the total-space curvature numerator on
	\(Ug,Vg\). We recall its derivation to fix the right-invariant signs;
	compare \cite{Puttmann}.
	
	For the right-invariant fields \(U^R_g=Ug\), one has
	\([U^R,V^R]=-[U,V]^R\). Koszul's formula gives
	\[
	\nabla_{U^R}V^R=-\Lambda(U,V)^R,\qquad
	\Lambda(U,V)=\frac12[U,V]+P^{-1}F_P(U,V).
	\]
	Metric compatibility says
	\[
	\langle\Lambda(U,V),W\rangle_P
	=-\langle V,\Lambda(U,W)\rangle_P.
	\]
	The right-trivialized curvature is therefore
	\[
	R_P(U,V)W
	=\Lambda(U,\Lambda(V,W))-\Lambda(V,\Lambda(U,W))
	-\Lambda([U,V],W).
	\]
	Put \(C=[U,V]\). Applying metric compatibility, the symmetry of \(F_P\),
	and \eqref{s3:analytic:adinvariance}, its three contractions are
	\[
	\begin{aligned}
		\langle\Lambda(U,\Lambda(V,V)),U\rangle_P
		&=-B(F_P(U,U),P^{-1}F_P(V,V)),\\
		-\langle\Lambda(V,\Lambda(U,V)),U\rangle_P
		&=B(F_P(U,V),P^{-1}F_P(U,V))-\frac14B(PC,C),\\
		-\langle\Lambda(C,V),U\rangle_P
		&=\frac12B([PU,V]+[U,PV],C)-\frac12B(PC,C).
	\end{aligned}
	\]
	For the middle line, write
	\(\Lambda(U,V)=C/2+P^{-1}F_P(U,V)\) and
	\(\Lambda(V,U)=-C/2+P^{-1}F_P(U,V)\); the mixed terms cancel.
	For the last line, substitute \(C,V,U\) into the formula for
	\(\Lambda\) and move brackets using
	\eqref{s3:analytic:adinvariance}.
	Their sum proves \eqref{s3:analytic:invariantcurvature}.
	In particular, \(P=I\) gives \(N_I(U,V)=|[U,V]|^2/4\).
	
	\subsection{The exact moving-plane expansion}
	
	Since \(PX=x\) and \(PY=y\), direct substitution gives
	\begin{equation}\label{s3:analytic:substitutions}
		\begin{gathered}
			[X,Y]=c+\varepsilon a+\varepsilon^2b,\qquad
			[PX,Y]+[X,PY]=2c+\varepsilon a,\\
			F_P(X,Y)=\frac{\varepsilon}{2}h,\qquad
			F_P(X,X)=\varepsilon f_x,\qquad
			F_P(Y,Y)=\varepsilon f_y.
		\end{gathered}
	\end{equation}
	Introduce the Lie algebra vector
	\[
	\mathcal W=b-Sa+S^2c,\qquad d=c+\varepsilon(a-Sc).
	\]
	The purpose of \(\mathcal W\) is to retain the dependence on \(P\)
	in one exact remainder:
	\[
	[X,Y]=Qd+\varepsilon^2\mathcal W.
	\]
	Using \(PQ=I\) and self-adjointness,
	\begin{equation}\label{s3:analytic:resolventexpansion}
		B(P[X,Y],[X,Y])
		=B(Qd,d)+2\varepsilon^2B(d,\mathcal W)
		+\varepsilon^4B(P\mathcal W,\mathcal W).
	\end{equation}
	The polynomial part of this expression is
	\[
	\begin{aligned}
		&|c|^2+\varepsilon\{2B(c,a)-B(Sc,c)\}\\
		&\quad+\varepsilon^2\{
		|a|^2+|Sc|^2+2B(c,b)-2B(a,Sc)\}\\
		&\quad+\varepsilon^3\{
		2B(a,b)-B(Sa,a)+2B(Sa,Sc)
		-2B(Sc,b)-B(S^2c,Sc)\}.
	\end{aligned}
	\]
	For example, the cubic coefficient before simplification is
	\(B(S(a-Sc),a-Sc)+2B(a-Sc,\mathcal W)\);
	expanding it gives the last line.
	
	One algebraic cancellation is needed at second order. Set
	\(A_1=[Sx,y]\), \(A_2=[x,Sy]\), so that \(a=A_1+A_2\) and
	\(h=A_1-A_2\). Jacobi's identity gives
	\[
	\begin{aligned}
		B(c,b)
		&=B\bigl(x,[y,[Sx,Sy]]\bigr)\\
		&=B\bigl(x,[[y,Sx],Sy]+[Sx,[y,Sy]]\bigr)\\
		&=B(A_1,A_2)+B(f_x,f_y).
	\end{aligned}
	\]
	Together with \(|h|^2-|a|^2=-4B(A_1,A_2)\), this proves
	\begin{equation}\label{s3:analytic:quadraticcancellation}
		-\frac14|a|^2+\frac14|h|^2-B(f_x,f_y)=-B(c,b).
	\end{equation}
	
	The first term of \eqref{s3:analytic:invariantcurvature} becomes
	\(\frac12B(2c+\varepsilon a,c+\varepsilon a+\varepsilon^2b)\).
	Its last two terms become
	\[
	\varepsilon^2\left(\frac14|h|^2-B(f_x,f_y)\right)
	+\varepsilon^3\left(\frac14B(Sh,h)-B(Sf_x,f_y)\right).
	\]
	Combining these expressions with
	\eqref{s3:analytic:resolventexpansion} and
	\eqref{s3:analytic:quadraticcancellation} proves the exact identity
	\begin{equation}\label{s3:analytic:ambient}
		\begin{aligned}
			N_\varepsilon^{\rm amb}={}&
			\frac14|c|^2+\frac34\varepsilon B(Sc,c)\\
			&+\varepsilon^2\left\{
			-\frac32B(c,b)+\frac32B(a,Sc)-\frac34|Sc|^2\right\}\\
			&+\varepsilon^3(\delta_3+R_3)
			-\frac34\varepsilon^4B(P\mathcal W,\mathcal W),\\
			R_3={}&\frac32B(Sc,b)-\frac32B(Sa,Sc)
			+\frac34B(S^2c,Sc).
		\end{aligned}
	\end{equation}
	The scalar \(R_3\) vanishes when \(c=0\). In particular,
	\begin{equation}\label{s3:analytic:ambientcommuting}
		c=0\quad\Longrightarrow\quad
		N_\varepsilon^{\rm amb}
		=\varepsilon^3\delta_3
		-\frac34\varepsilon^4B(P(b-Sa),b-Sa).
	\end{equation}
	This is the inverse-linear expansion of \cite{HT} with deformation
	operator \(\Psi=-S\), derived here in the conventions of this paper.
	
	\subsection{The vertical bracket and the quotient curvature}
	
	We now compute the term added by the Riemannian submersion.
	For a tangent vector \(Ug\), define
	\[
	\lambda_g(Ug)=J_g^*PU,\qquad M=J_g^*PJ_g.
	\]
	Thus \(\lambda\) is an \(\im\HH\)-valued one-form, and \(M\) is a
	positive definite operator on \(\im\HH\). In an orthonormal basis of
	\(\im\HH\), \(M\) is the Gram matrix of the vertical vectors
	\((J_g\xi)g\). The form \(\lambda_g\) records inner products with
	those vectors. In particular,
	\[
	\lambda_g((J_g\xi)g)=M\xi,\qquad
	\ker\lambda_g=\mathcal H_{\varepsilon,g}.
	\]
	
	Let \(\overline X,\overline Y\) be local horizontal extensions of
	\(Xg,Yg\), and put
	\[
	\beta=(d\lambda)_g(Xg,Yg)\in\im\HH.
	\]
	Since \(\lambda(\overline X)=\lambda(\overline Y)=0\), the exterior
	derivative identity gives
	\(\beta=-\lambda_g([\overline X,\overline Y]_g)\).
	Writing the vertical part of the bracket as \((J_g\xi)g\) therefore
	gives \(M\xi=-\beta\). It follows that
	\[
	[\overline X,\overline Y]_g^{\mathcal V}
	=(-J_gM^{-1}\beta)g,\qquad
	\bigl|[\overline X,\overline Y]_g^{\mathcal V}\bigr|_{\widetilde g_\varepsilon}^2
	=\beta\cdot M^{-1}\beta.
	\]
	O'Neill's formula, as presented in Petersen's textbook
	\cite[Theorem~4.5.3]{ONeill}, now reads
	\begin{equation}\label{s3:analytic:oneill}
		N_\varepsilon=N_\varepsilon^{\rm amb}
		+\frac34\beta\cdot M^{-1}\beta.
	\end{equation}
	
	To compute \(\beta\), we may use constant right-invariant extensions:
	the exterior derivative is tensorial, even though these extensions need
	not remain horizontal. Along \(g(t)=e^{tX}g\),
	\[
	\left.\frac{d}{dt}\right|_{t=0}E_{g(t)}\xi=[X,E_g\xi],
	\qquad
	\left.\frac{d}{dt}\right|_{t=0}J_{g(t)}\xi=-[X,E_g\xi].
	\]
	Together with \([X^R,Y^R]=-[X,Y]^R\), this gives, for every
	\(\xi\in\im\HH\),
	\[
	\begin{aligned}
		\xi\cdot\beta
		&=-B(PY,[X,E_g\xi])+B(PX,[Y,E_g\xi])
		+B(P[X,Y],J_g\xi)\\
		&=B([X,PY]+[PX,Y],E_g\xi)+B(P[X,Y],J_g\xi).
	\end{aligned}
	\]
	Consequently,
	\[
	\beta=E_g^*([X,PY]+[PX,Y])+J_g^*P[X,Y].
	\]
	Using
	\[
	P[X,Y]=c+\varepsilon(a-Sc)+\varepsilon^2P\mathcal W
	\]
	and \(E_g+J_g=D\), we obtain the exact formula
	\begin{equation}\label{s3:analytic:beta}
		\boxed{\beta=(D+E_g)^*c
			+\varepsilon(\zeta-J_g^*Sc)
			+\varepsilon^2J_g^*P\mathcal W.}
	\end{equation}
	
	At \(\varepsilon=0\), equations
	\eqref{s3:analytic:ambient} and \eqref{s3:analytic:oneill} give
	\begin{equation}\label{s3:analytic:baseline}
		N_0=\frac14|c|^2+
		\frac34\bigl((D+E_g)^*c\bigr)\cdot
		(J_g^*J_g)^{-1}\bigl((D+E_g)^*c\bigr).
	\end{equation}
	Thus \(N_0=0\) exactly when \(c=0\).
	For a fixed commuting pair,
	\(\beta=\varepsilon\zeta+O(\varepsilon^2)\), so
	\[
	N_\varepsilon
	=\frac34\varepsilon^2\,
	\zeta\cdot(J_g^*J_g)^{-1}\zeta+O(\varepsilon^3).
	\]
	If also \(\zeta=0\), then \(\beta=O(\varepsilon^2)\), and
	\[
	N_\varepsilon=\varepsilon^3\delta_3+O(\varepsilon^4).
	\]
	These expansions explain the definition of \(\mathcal Z\).
	The following argument supplies the uniform estimate needed when the
	point and the plane vary with \(\varepsilon\).
	
	\subsection{A uniform estimate with a fourth-order error}
	
	\begin{proof}[Proof of Proposition~\ref{rw:uniform-proposition}]
		All bounds in this proof are independent of \((g,x,y)\in\mathcal I\).
		Bilinearity of the bracket on the finite-dimensional space
		\(\mathfrak g\) gives a constant \(C_{\rm br}\) such that
		\[
		|[U,V]|\le C_{\rm br}|U|\,|V|.
		\]
		For example, one can take the maximum of \(|[U,V]|\) over the
		product of the two unit spheres. Since \(|x|=|y|=1\) and \(S\) is
		fixed, the vectors \(c,a,b,h,f_x,f_y,\mathcal W\) and the scalar
		\(\delta_3\) are uniformly bounded.
		
		Conjugation by \(g\) preserves \(B\), and hence
		\[
		|D\xi|=|\xi|,\qquad |E_g\xi|=\frac{|\xi|}{\sqrt2}.
		\]
		The triangle inequalities imply
		\[
		\left(1-\frac1{\sqrt2}\right)|\xi|
		\le |J_g\xi|\le2|\xi|.
		\]
		Together with \eqref{rw:metric-comparison}, this gives
		\begin{equation}\label{s3:analytic:Mbounds}
			\frac23\left(1-\frac1{\sqrt2}\right)^2 I\le M\le8I,
			\qquad M^{-1}\ge\frac18I.
		\end{equation}
		In particular, \(M\) stays uniformly invertible.
		
		Set
		\[
		L_\varepsilon=(D+E_g)^*-\varepsilon J_g^*S.
		\]
		Both \(L_\varepsilon\) and \(J_g^*P\mathcal W\) are uniformly
		bounded for \(0\le\varepsilon\le\varepsilon_*\).
		Let \(t=|c|\). Each term in \(R_3\) in
		\eqref{s3:analytic:ambient} has at least one factor of \(c\).
		Cauchy--Schwarz, the preceding bounds, and
		\(\varepsilon_*\le1\) therefore give a single constant \(K\ge1\),
		chosen large enough that
		\begin{equation}\label{rw:coarse-bounds}
			\begin{gathered}
				|\delta_3|\le K,\qquad
				\|L_\varepsilon\|\le K,\qquad |J_g^*P\mathcal W|\le K,\\
				N_\varepsilon^{\rm amb}\ge
				\left(\frac14-K\varepsilon\right)t^2
				-K\varepsilon^2t+\varepsilon^3\delta_3-K\varepsilon^4.
			\end{gathered}
		\end{equation}
		To verify the last inequality directly, the first-order term is bounded
		by a constant times \(\varepsilon t^2\), the second-order coefficient
		by a constant times \(t+t^2\), and \(R_3\) by a constant times
		\(t+t^2\). The exact last term is bounded by a constant times
		\(\varepsilon^4\), since \(P\) and \(\mathcal W\) are bounded.
		Absorbing \(\varepsilon^2t^2,\varepsilon^3t^2\) into
		\(K\varepsilon t^2\), and \(\varepsilon^3t\) into
		\(K\varepsilon^2t\), proves the claim.
		
		Choose
		\[
		0<\varepsilon_1\le\min\{\varepsilon_*,(8K)^{-1}\}.
		\]
		For \(0<\varepsilon\le\varepsilon_1\), the coefficient of \(t^2\)
		in \eqref{rw:coarse-bounds} is at least \(1/8\).
		The square
		\[
		0\le\left(\frac{t}{4}-2K\varepsilon^2\right)^2
		=\frac{t^2}{16}-K\varepsilon^2t+4K^2\varepsilon^4
		\]
		then gives
		\begin{equation}\label{s3:analytic:ambientlower}
			N_\varepsilon^{\rm amb}\ge
			\frac1{16}|c|^2+\varepsilon^3\delta_3
			-(K+4K^2)\varepsilon^4.
		\end{equation}
		In particular, the absorption leaves the cubic coefficient unchanged.
		
		For the O'Neill term, \eqref{s3:analytic:beta} gives
		\[
		\bigl|\beta-(L_\varepsilon c+\varepsilon\zeta)\bigr|
		\le K\varepsilon^2.
		\]
		For any Euclidean vectors \(u,v\),
		\[
		|u+v|^2-\frac12|u|^2+|v|^2
		=\frac12|u+2v|^2\ge0.
		\]
		Applying this identity and \eqref{s3:analytic:Mbounds} yields
		\begin{equation}\label{s3:analytic:oneilllower}
			\frac34\beta\cdot M^{-1}\beta
			\ge\frac3{32}|\beta|^2
			\ge\frac3{64}|L_\varepsilon c+\varepsilon\zeta|^2
			-\frac{3K^2}{32}\varepsilon^4.
		\end{equation}
		To allow cancellation between \(L_\varepsilon c\) and
		\(\varepsilon\zeta\), use their defining relation in the form
		\[
		\varepsilon^2|\zeta|^2
		\le2|L_\varepsilon c+\varepsilon\zeta|^2+2K^2|c|^2.
		\]
		It follows that
		\[
		\frac1{16}|c|^2+\frac3{64}|L_\varepsilon c+\varepsilon\zeta|^2
		\ge a_0\bigl(|c|^2+\varepsilon^2|\zeta|^2\bigr),
		\qquad
		a_0=\min\left\{\frac1{16(1+2K^2)},\frac3{128}\right\}>0.
		\]
		Adding \eqref{s3:analytic:ambientlower} and
		\eqref{s3:analytic:oneilllower} proves \eqref{s3:eq:uniform},
		with any
		\[
		C_0\ge K+4K^2+\frac{3K^2}{32}.
		\]
		This choice also gives \(|\delta_3|\le C_0\).
		Every bound is uniform, and the two square completions are valid even
		when \(c\) or \(\zeta\) vanishes.
	\end{proof}
	
	\section{The horizontal restriction}
	\label{s3:detail:spectral}
	
	The analytic estimate reduces the problem to \(\delta_3>0\) on the
	critical set \(\mathcal Z\). We now extract the geometric restriction
	needed for the algebraic proof.
	For \(x=(r;u_0,u_1,u_2)\), \(y=(s;v_0,v_1,v_2)\), put
	\[
	H(x,y)=|D^*x|^2+|D^*y|^2=|u_0|^2+|v_0|^2,
	\]
	where \(D^*\) is the adjoint in \eqref{s3:detail:adjoints}.
	Here \(B\) and the coordinates are those of \eqref{rw:coordinates}.
	We will show that every horizontal commuting orthonormal pair satisfies
	\(1/4\le H\le1/2\). The forcing condition \(\zeta=0\) is not needed
	for this bound.
	
	\subsection{Commuting matrices and changes of frame}
	
	\begin{lemma}\label{s3:detail:simdiag}
		If \(Y_1,Y_2\in\mathfrak{sp}(2)\) commute, there exist \(A\in\Sp(2)\)
		and real numbers \(a_1,a_2,b_1,b_2\) such that
		\[
		Y_1=A\diag(ia_1,ia_2)A^*,\qquad
		Y_2=A\diag(ib_1,ib_2)A^*.
		\]
	\end{lemma}
	
	\begin{proof}
		
		This follows from the fact that every maximal
		abelian subalgebra of $\mathfrak{sp}(2)$ is conjugate to the diagonal one, see
		\cite[Theorem~11.9]{Hall}.
		We give a direct matrix proof here for completeness.
		
		Write a quaternionic matrix as \(Z+Wj\), with complex matrices \(Z,W\).
		The representation
		\[
		\rho(Z+Wj)=
		\begin{pmatrix}Z&W\\-\overline W&\overline Z\end{pmatrix}
		\]
		preserves products and conjugate transpose. Its image commutes with
		the anti-linear isometry
		\[
		\mathcal J(v,w)=(\overline w,-\overline v),\qquad
		\mathcal J^2=-I
		\]
		of \(\C^4\). The commuting skew-Hermitian matrices
		\(\rho(Y_1),\rho(Y_2)\) admit a common orthonormal eigenbasis:
		diagonalize the second matrix inside each eigenspace of the first.
		If \(\xi\) has joint eigenvalues \(ia,ib\), then \(\mathcal J\xi\)
		has joint eigenvalues \(-ia,-ib\).
		
		Choose the eigenbasis in orthonormal pairs
		\(\xi_1,\mathcal J\xi_1,\xi_2,\mathcal J\xi_2\).
		For a nonzero joint eigenvalue pair, choose a basis in one joint
		eigenspace and take its \(\mathcal J\)-images in the opposite one.
		In the common zero eigenspace, \(\xi\perp\mathcal J\xi\), and the
		orthogonal complement of their span is again \(\mathcal J\)-invariant;
		this gives the remaining pairs.
		Writing the first two columns as \((Z,-\overline W)^t\), the unitary
		matrix
		\[
		(\xi_1,\xi_2,-\mathcal J\xi_1,-\mathcal J\xi_2)
		=\rho(Z+Wj)
		\]
		is the image of a matrix \(A\in\Sp(2)\).
		It diagonalizes both matrices in the stated form.
	\end{proof}
	
	\begin{lemma}\label{s3:detail:H-invariance}
		Let \(x,y\) be \(B\)-orthonormal, and let \(O\in O(2)\).
		Set
		\[
		\widehat x=O_{11}x+O_{12}y,\qquad
		\widehat y=O_{21}x+O_{22}y.
		\]
		Then \(\widehat x,\widehat y\) are an orthonormal basis of the same
		plane, \(H(\widehat x,\widehat y)=H(x,y)\), and
		\[
		[\widehat x,\widehat y]=\det(O)[x,y],\qquad
		\zeta(\widehat x,\widehat y)=\det(O)\zeta(x,y).
		\]
		For a fixed \(g\), the equations \(J_g^*x=J_g^*y=0\) are also preserved.
	\end{lemma}
	
	\begin{proof}
		For any real-linear map \(L\), orthogonality of \(O\) gives
		\[
		|L\widehat x|^2+|L\widehat y|^2
		=\sum_{j,k=1}^2(O^tO)_{jk}\langle Lx_j,Lx_k\rangle
		=|Lx|^2+|Ly|^2,
		\qquad x_1=x,\ x_2=y.
		\]
		Take \(L=D^*\) for the assertion about \(H\).
		The bracket and
		\(\zeta(x,y)=D^*([Sx,y]+[x,Sy])\) are alternating bilinear maps,
		so each transforms by \(\det(O)\).
		The horizontal equations are preserved by linearity of \(J_g^*\).
	\end{proof}
	
	\subsection{The moment bound}
	
	\begin{proposition}\label{rw:moment-bound}
		Suppose \((g,x,y)\in\mathcal I\) and \([x,y]=0\).
		Then
		\begin{equation}\label{s3:detail:H}
			\frac14\le H(x,y)\le\frac12.
		\end{equation}
	\end{proposition}
	
	\begin{proof}
		By Lemma~\ref{s3:detail:simdiag}, the pair lies in a common
		two-dimensional diagonal torus conjugated by a matrix \(A\in\Sp(2)\).
		Since the pair is orthonormal, its real span is that entire torus.
		An orthogonal change of basis in this plane therefore gives
		\[
		\widehat x=A\diag(\sqrt2\,i,0)A^*,\qquad
		\widehat y=A\diag(0,\sqrt2\,i)A^*.
		\]
		Lemma~\ref{s3:detail:H-invariance} preserves both \(H\) and
		horizontality at the original point \(g\).
		
		Let \(e=(1,0)^t\in\HH^2\), let \(\mathbf w_g=ge\) be the first
		column of \(g\), and write
		\[
		(q_1,q_2)^t=A^*\mathbf w_g,\qquad q_1,q_2\in\HH.
		\]
		Unitarity gives \(|q_1|^2+|q_2|^2=1\).
		The equality \(D^*=E_g^*\) on horizontal vectors and
		\eqref{s3:detail:adjoints} imply
		\[
		D^*\widehat x=\frac1{\sqrt2}\overline q_1\,i\,q_1,\qquad
		D^*\widehat y=\frac1{\sqrt2}\overline q_2\,i\,q_2.
		\]
		Quaternionic norm is multiplicative. Hence, with
		\(t=|q_1|^2\in[0,1]\),
		\[
		H=\frac12(|q_1|^4+|q_2|^4)
		=\frac12\bigl(t^2+(1-t)^2\bigr)
		=\frac14+\left(t-\frac12\right)^2.
		\]
		This proves both bounds, including their endpoints.
	\end{proof}
	
	The diagonalization in this proof is used only to select a basis of the
	same commuting plane. The point \(g\) and its horizontal map \(J_g^*\)
	are kept fixed. We have therefore proved that the projection
	\((g,x,y)\mapsto(x,y)\) sends \(\mathcal Z\) into the algebraic set
	\begin{equation}\label{s3:hand:raw}
		\mathcal A=\left\{(x,y):
		|x|=|y|=1,\ B(x,y)=0,\ [x,y]=0,\ \zeta(x,y)=0,\
		\frac14\le H(x,y)\le\frac12\right\}.
	\end{equation}
	The next section proves the cubic inequality on all of \(\mathcal A\).
	Only this inclusion is required; an arbitrary pair in \(\mathcal A\)
	need not be realized as a horizontal pair at some point of \(G\).
	
	\section{The algebraic critical inequality}
	\label{s3:hand:critical}
	
	We prove the lower bound for \(\delta_3\) on the set \(\mathcal A\)
	defined in \eqref{s3:hand:raw}. Thus
	\(x=(r;u_0,u_1,u_2)\), \(y=(s;v_0,v_1,v_2)\) are
	\(B\)-orthonormal, \([x,y]=\zeta(x,y)=0\), and
	\(H=|u_0|^2+|v_0|^2\in[1/4,1/2]\).
	The bracket expressions and cubic scalar are defined in
	\eqref{s3:eq:coeffs}--\eqref{s3:eq:delta}, with
	\(S=S_3\). No group point or horizontal map is used in the
	algebraic argument below.
	
	There are four reductions: the real coordinates vanish; the six
	imaginary vectors lie in a common plane; a Hermitian matrix gives
	relations among their scalar invariants; and the cubic coefficient
	becomes a sum of nonnegative terms with a positive lower bound.
	
	For later substitutions, quaternion multiplication gives
	\begin{equation}\label{s3:eq:bracket}
		\begin{aligned}
			\tfrac12[x,y]_r&=u_2\cdot v_1-u_1\cdot v_2,\\
			\tfrac12[x,y]_0&=\sum_{j=0}^2u_j\times v_j,\\
			\tfrac12[x,y]_1&=u_0\times v_1+u_1\times v_0+rv_2-su_2,\\
			\tfrac12[x,y]_2&=u_0\times v_2+u_2\times v_0+su_1-rv_1.
		\end{aligned}
	\end{equation}
	The subscripts denote the four coordinate blocks in
	\eqref{rw:coordinates}; the first is real and the other three are
	vectors in \(\im\HH\). Substitution of \(Sx,Sy\) into the block
	numbered \(0\) gives
	\begin{equation}\label{s3:detail:force}
		\zeta(x,y)=4\bigl\{(u_0-u_2)\times v_1+
		u_1\times(v_0-v_2)\bigr\}.
	\end{equation}
	In particular, \(\zeta\) does not depend on the real coordinates.
	
	\subsection{The real coordinates vanish}
	\label{s3:scalar:section}
	
	We first record the spectral alternatives used in the proof.
	
	\begin{lemma}\label{s3:detail:centralizer}
		Let \(Z\in\mathfrak{sp}(2)\) have \(B(Z,Z)=1\).
		Exactly one of the following occurs:
		\begin{enumerate}
			\item \(Z\) is conjugate to \(\diag(i\lambda_1,i\lambda_2)\) with
			\(\lambda_1\lambda_2\ne0\) and
			\(|\lambda_1|\ne|\lambda_2|\). Its centralizer is
			\(\Span_{\R}\{Z,Z^3\}\).
			\item \(Z^2=-I\).
			\item \(Z\) has quaternionic rank one and \(Z^3=-2Z\).
		\end{enumerate}
	\end{lemma}
	
	\begin{proof}
		Lemma~\ref{s3:detail:simdiag} diagonalizes \(Z\).
		Its normalization gives \(\lambda_1^2+\lambda_2^2=2\).
		Equal absolute values therefore give the second case, and one zero
		eigenvalue gives the third.
		In the first case, a commuting matrix
		\(\left(\begin{smallmatrix}a&q\\-\overline q&b\end{smallmatrix}\right)\)
		satisfies
		\[
		[i\lambda_1,a]=[i\lambda_2,b]=0,\qquad
		i\lambda_1q=qi\lambda_2.
		\]
		Taking norms in the last equation gives \(q=0\); the first two make
		\(a,b\) real multiples of \(i\).
		The centralizer is therefore the two-dimensional diagonal torus.
		The determinant
		\[
		\det\begin{pmatrix}\lambda_1&-\lambda_1^3\\
			\lambda_2&-\lambda_2^3\end{pmatrix}
		=\lambda_1\lambda_2(\lambda_1^2-\lambda_2^2)
		\]
		is nonzero, so \(Z,Z^3\) span it.
	\end{proof}
	
	\begin{lemma}\label{s3:hand:scalar}
		Suppose \(x=(r;u_i)\), \(y=(s;v_i)\) are \(B\)-orthonormal,
		\([x,y]=\zeta(x,y)=0\), and \(H(x,y)<1\).
		Then \(r=s=0\).
	\end{lemma}
	
	\begin{proof}
		If \(R_0=(r^2+s^2)^{1/2}>0\), replace the pair by
		\[
		x'=\frac{rx+sy}{R_0},\qquad
		y'=\frac{-sx+ry}{R_0}.
		\]
		Lemma~\ref{s3:detail:H-invariance} preserves all hypotheses.
		Relabeling the pair, we may assume
		\begin{equation}\label{s3:scalar:normalize}
			r\ne0,\qquad s=0.
		\end{equation}
		We obtain a contradiction in each case of
		Lemma~\ref{s3:detail:centralizer} applied to \(y\).
		
		We will use two matrix-power identities. Since \(|y|=1\) and
		\(y=(0;v_0,v_1,v_2)\), direct multiplication gives
		\begin{equation}\label{s3:scalar:square}
			y^2=
			\begin{pmatrix}
				-1-2v_0\cdot v_1&-2v_0\cdot v_2+2v_1\times v_2\\
				-2v_0\cdot v_2-2v_1\times v_2&-1+2v_0\cdot v_1
			\end{pmatrix}.
		\end{equation}
		Writing \(y^3=(r_3;\widetilde v_0,\widetilde v_1,\widetilde v_2)\),
		one further multiplication gives
		\begin{equation}\label{s3:scalar:cube}
			\begin{aligned}
				r_3&=-2v_0\cdot(v_1\times v_2),\\
				\widetilde v_0&=-v_0-2(v_0\cdot v_1)v_1-2(v_0\cdot v_2)v_2,\\
				\widetilde v_1&=-2(v_0\cdot v_1)v_0
				-(1+2|v_2|^2)v_1+2(v_1\cdot v_2)v_2,\\
				\widetilde v_2&=-2(v_0\cdot v_2)v_0
				+2(v_1\cdot v_2)v_1-(1+2|v_1|^2)v_2.
			\end{aligned}
		\end{equation}
		The scalar entry is obtained from the real part of the upper-right
		entry; the vector entries follow from
		\(uv=-u\cdot v+u\times v\) and the vector triple-product identity.
		
		\emph{Regular case.}
		Here \(x=\lambda y+\mu y^3\) for real \(\lambda,\mu\).
		Since \(x\) has nonzero real coordinate and \(y\) has zero real
		coordinate, \eqref{s3:scalar:cube} implies
		\(\mu\ne0\) and
		\(\Delta=v_0\cdot(v_1\times v_2)\ne0\).
		The alternating property of \(\zeta\) gives
		\(\zeta(y,y^3)=0\).
		Substituting \eqref{s3:scalar:cube} into
		\eqref{s3:detail:force} yields
		\begin{equation}\label{s3:scalar:forcingcube}
			\begin{aligned}
				\frac18\zeta(y,y^3)={}&
				-(|v_2|^2+v_0\cdot v_2)\,v_0\times v_1\\
				&+(v_1\cdot v_2-v_0\cdot v_1)\,v_0\times v_2\\
				&-(|v_2|^2-|v_1|^2+v_0\cdot v_2)\,v_1\times v_2.
			\end{aligned}
		\end{equation}
		Dotting this zero vector first with \(v_2\), then with \(v_0\),
		and using \(\Delta\ne0\), gives
		\[
		|v_2|^2+v_0\cdot v_2=0,\qquad
		|v_2|^2-|v_1|^2+v_0\cdot v_2=0.
		\]
		Subtracting forces \(v_1=0\), contrary to \(\Delta\ne0\).
		
		\emph{The case \(y^2=-I\).}
		Equation~\eqref{s3:scalar:square} gives
		\(v_0\cdot v_1=v_0\cdot v_2=0\) and \(v_1\times v_2=0\).
		Choose an oriented orthonormal basis \(i,j,k\) of \(\im\HH\) such that
		\[
		v_0=di,\qquad v_1=ej,\qquad v_2=fj,\qquad
		u_\nu=\alpha_\nu i+\beta_\nu j+\gamma_\nu k.
		\]
		This includes all vanishing-coordinate cases.
		The \(j\)-components of \([x,y]_1/2=0\) and
		\(\zeta(x,y)/4=0\), respectively, give
		\[
		d\gamma_1+rf=0,\qquad d\gamma_1=0.
		\]
		Thus \(f=0\). The \(i\)-components of the same equations now give
		\(-e\gamma_0=0\) and \(-e\gamma_0+e\gamma_2=0\).
		If \(e\ne0\), then \(\gamma_0=\gamma_2=0\), whereas the
		\(j\)-component of \([x,y]_2/2=0\) gives
		\(d\gamma_2-re=0\), a contradiction.
		Thus \(e=f=0\). Normalization gives \(|v_0|=1\), contradicting
		\(H<1\).
		
		\emph{The rank-one case.}
		Now \(y^3=-2y\). Equation~\eqref{s3:scalar:cube} gives
		\(\det(v_0,v_1,v_2)=0\), so the three vectors lie in a common plane.
		Choose coordinates
		\[
		v_0=di,\qquad v_1=ei+\sigma j,\qquad v_2=fi+\tau j,
		\qquad d\ge0,\qquad
		u_\nu=\alpha_\nu i+\beta_\nu j+\gamma_\nu k.
		\]
		If \(v_0=0\), choose any plane containing \(v_1,v_2\) and set \(d=0\).
		
		The two eigenvalues of \(y^2\) are \(-2,0\), whence
		\(\frac12\Rea\tr_{\HH}(y^4)=2\).
		Squaring \eqref{s3:scalar:square} and taking its real trace gives
		\[
		\frac12\Rea\tr_{\HH}(y^4)
		=1+4\bigl\{(v_0\cdot v_1)^2+(v_0\cdot v_2)^2
		+|v_1\times v_2|^2\bigr\}.
		\]
		Together with normalization, this becomes
		\[
		d^2+e^2+f^2+\sigma^2+\tau^2=1,\qquad
		d^2(e^2+f^2)+(e\tau-f\sigma)^2=\frac14.
		\]
		For the real vectors \(\mathbf a=(d,\sigma,\tau)\),
		\(\mathbf b=(0,e,f)\), these equations say
		\(|\mathbf a|^2+|\mathbf b|^2=1\) and
		\(|\mathbf a\times\mathbf b|^2=1/4\).
		Equality holds throughout
		\[
		|\mathbf a\times\mathbf b|^2\le|\mathbf a|^2|\mathbf b|^2
		\le\frac14(|\mathbf a|^2+|\mathbf b|^2)^2.
		\]
		Hence \(|\mathbf a|^2=|\mathbf b|^2=1/2\) and
		\(\mathbf a\cdot\mathbf b=0\).
		There is a real number \(\theta\) such that
		\begin{equation}\label{s3:scalar:rankoneparameters}
			e^2+f^2=\frac12,\qquad
			\sigma=-\theta f,\quad \tau=\theta e,\qquad
			d^2=\frac{1-\theta^2}{2},\quad |\theta|\le1.
		\end{equation}
		
		For clarity, the components of the equations used next are
		\begin{equation}\label{s3:scalar:singularbrackets}
			\begin{aligned}
				-\sigma\gamma_1-\tau\gamma_2&=0,
				&d\gamma_0+e\gamma_1+f\gamma_2&=0,\\
				-\sigma\gamma_0+rf&=0,
				&e\gamma_0+d\gamma_1+r\tau&=0,\\
				-\tau\gamma_0-re&=0,
				&f\gamma_0+d\gamma_2-r\sigma&=0.
			\end{aligned}
		\end{equation}
		These are the \(i,j\)-components of the three imaginary blocks of
		\([x,y]/2=0\). The corresponding components of \(\zeta/4=0\) are
		\begin{equation}\label{s3:scalar:singularforcing}
			-\sigma\gamma_0+\tau\gamma_1+\sigma\gamma_2=0,\qquad
			e\gamma_0+(d-f)\gamma_1-e\gamma_2=0.
		\end{equation}
		Substituting \eqref{s3:scalar:rankoneparameters} into the left
		equations in the last two rows of
		\eqref{s3:scalar:singularbrackets} gives
		\[
		f(\theta\gamma_0+r)=0,\qquad
		-e(\theta\gamma_0+r)=0.
		\]
		Since \(e^2+f^2=1/2\), we obtain
		\(\theta\gamma_0=-r\). In particular,
		\begin{equation}\label{s3:scalar:tnonzero}
			\theta\ne0,\qquad \gamma_0=-r/\theta\ne0.
		\end{equation}
		
		If \(d=0\), the first row of
		\eqref{s3:scalar:singularbrackets} is
		\[
		\theta f\gamma_1-\theta e\gamma_2=0,\qquad
		e\gamma_1+f\gamma_2=0.
		\]
		Its determinant is \(\theta(e^2+f^2)=\theta/2\ne0\), so
		\(\gamma_1=\gamma_2=0\).
		Equation~\eqref{s3:scalar:singularforcing} then gives
		\(\sigma\gamma_0=e\gamma_0=0\), and hence \(\sigma=e=0\).
		Since \(\sigma=-\theta f\), also \(f=0\), a contradiction.
		
		Suppose \(d>0\). The right equations in the last two rows of
		\eqref{s3:scalar:singularbrackets}, together with
		\(1-\theta^2=2d^2\), give
		\[
		\gamma_1=\frac{2red}{\theta},\qquad
		\gamma_2=\frac{2rfd}{\theta}.
		\]
		Substituting into the forcing equations produces
		\[
		\begin{aligned}
			0&=-rf+2re^2d-2rf^2d
			=r\{-f+d(1-4f^2)\},\\
			0&=\frac{re}{\theta}(-1+2d^2-4df)
			=-\frac{re}{\theta}(\theta^2+4df).
		\end{aligned}
		\]
		Division by \(r\ne0\) and \(\theta\ne0\) is legitimate, giving
		\begin{equation}\label{s3:scalar:finalequations}
			f=d(1-4f^2),\qquad e(\theta^2+4df)=0.
		\end{equation}
		If \(e=0\), then \(f^2=1/2\) and the first equation gives \(f=-d\).
		Thus \(d^2=1/2\), forcing \(\theta=0\), a contradiction.
		If \(e\ne0\), then \(\theta^2=-4df\), and the first equation gives
		\[
		f=d+\theta^2f,\qquad f(1-\theta^2)=d.
		\]
		Since \(d>0\) and \(1-\theta^2=2d^2>0\), this implies \(f>0\),
		contradicting \(\theta^2=-4df\).
		All cases contradict \eqref{s3:scalar:normalize}.
	\end{proof}
	
	Only \(H<1\) was needed for the scalar-coordinate lemma. In particular
	it applies to every pair in \(\mathcal A\); the lower bound
	\(H\ge1/4\) will enter only in the final positivity estimate.
	
	\subsection{A common plane for the imaginary components}
	
	The scalar-coordinate lemma gives \(r=s=0\) on \(\mathcal A\).
	Define
	\[
	F_{ij}=u_i\times v_j+u_j\times v_i\qquad(i\ne j).
	\]
	The imaginary blocks numbered \(1,2\) of \([x,y]=0\) give
	\(F_{01}=F_{02}=0\), and \eqref{s3:detail:force} gives
	\(F_{12}=0\). The remaining blocks of the commutator give
	\begin{equation}\label{s3:hand:cross}
		F_{01}=F_{02}=F_{12}=0,\qquad
		\sum_{i=0}^2u_i\times v_i=0,\qquad
		u_1\cdot v_2=u_2\cdot v_1.
	\end{equation}
	
	\begin{lemma}\label{s3:hand:physical}
		For every \((x,y)\in\mathcal A\), the six vectors
		\(u_0,u_1,u_2,v_0,v_1,v_2\) lie in a common real subspace of
		\(\im\HH\) of dimension at most two.
	\end{lemma}
	
	\begin{proof}
		We distinguish the dimension of \(\Span_{\R}\{u_0,u_1,u_2\}\).
		
		If the dimension is three, write \(v_i=\sum_jm_{ij}u_j\). Since
		\[\det(u_0\times u_1,u_0\times u_2,u_1\times u_2)=
		\left(\det(u_0,u_1,u_2)\right)^2\neq0.
		\]
		We know that the cross products
		\(u_0\times u_1,u_0\times u_2,u_1\times u_2\) are independent.
		The coefficients of the three equations \(F_{ij}=0\) in this basis are
		\[
		\begin{array}{c|ccc}
			&u_0\times u_1&u_0\times u_2&u_1\times u_2\\ \hline
			F_{01}&m_{11}-m_{00}&m_{12}&m_{02}\\
			F_{02}&m_{21}&m_{22}-m_{00}&-m_{01}\\
			F_{12}&-m_{20}&-m_{10}&m_{22}-m_{11}
		\end{array}
		\]
		Thus \(m_{ij}=t\delta_{ij}\) for some real \(t\), and \(y=tx\),
		contrary to orthonormality.
		
		If the dimension is two, let \(n\) be a unit normal to this plane
		and write \(v_i=v_i^\parallel+\gamma_i n\).
		The component of \(F_{ij}=0\) in the plane is
		\[
		(\gamma_j u_i+\gamma_i u_j)\times n=0.
		\]
		Cross product with \(n\) is invertible on the plane, so
		\(\gamma_j u_i+\gamma_i u_j=0\).
		Choose independent \(u_i,u_j\). Their equation forces
		\(\gamma_i=\gamma_j=0\); pairing a nonzero one of these rows with
		the remaining row then gives the remaining \(\gamma_k=0\).
		All \(v_i\) lie in the same plane.
		
		If the dimension is one, write \(u_i=\alpha_i u\) with \(u\ne0\)
		and choose \(i\) such that \(\alpha_i\ne0\).
		For \(j\ne i\), the relation \(F_{ij}=0\) gives
		\[
		u\times(\alpha_i v_j+\alpha_j v_i)=0,\qquad
		v_j^\perp=-\frac{\alpha_j}{\alpha_i}v_i^\perp,
		\]
		where the superscript means projection to \(u^\perp\).
		All six vectors lie in
		\(\Span_{\R}\{u,v_i^\perp\}\).
		Dimension zero is impossible because \(r=0\) and \(|x|=1\).
	\end{proof}
	
	Choose an oriented Euclidean two-plane containing the six vectors,
	and an oriented orthonormal basis in it. Write their components as
	\(u_i=(u_{i1},u_{i2})\), \(v_i=(v_{i1},v_{i2})\), and put
	\[
	\omega(a,b)=a_1b_2-a_2b_1,\qquad
	z_i=u_{i1}+i u_{i2},\qquad w_i=v_{i1}+i v_{i2}.
	\]
	Here \(z,w\) are complex column vectors in \(\C^3\).
	If \(n\) is the oriented unit normal, then
	\(u_i\times v_j=\omega(u_i,v_j)n\).
	Normalization and \eqref{s3:hand:cross} give
	\begin{equation}\label{rw:complex-frame}
		z^*z=w^*w=1,\qquad
		z^*w=\sum_i u_i\cdot v_i+i\sum_i\omega(u_i,v_i)=0.
	\end{equation}
	Thus the real orthonormality of \(x,y\), together with the remaining
	cross-product equation, produces a complex orthonormal frame.
	
	\subsection{The scalar invariants and their relations}
	
	The complex frame permits us to collect the constraints in a Hermitian
	matrix with fixed spectrum. Define real numbers
	\[
	p=u_0\cdot v_1-u_1\cdot v_0,\qquad
	q=u_0\cdot v_2-u_2\cdot v_0,\qquad
	t_i=-2\omega(u_i,v_i)\quad(i=0,1,2),
	\]
	and the Hermitian operator
	\[
	K=-i(zw^*-wz^*)\quad\text{on }\C^3.
	\]
	The star here is complex conjugate transpose.
	Entrywise,
	\[
	(zw^*-wz^*)_{ij}
	=u_i\cdot v_j-u_j\cdot v_i
	-i\{\omega(u_i,v_j)+\omega(u_j,v_i)\}.
	\]
	The off-diagonal area terms vanish by \(F_{ij}=0\), the entry
	numbered \(12\) vanishes by \(u_1\cdot v_2=u_2\cdot v_1\), and
	the diagonal entries are \(it_i\). Hence
	\begin{equation}\label{s3:hand:K}
		K=\begin{pmatrix}
			t_0&-ip&-iq\\ ip&t_1&0\\ iq&0&t_2
		\end{pmatrix}.
	\end{equation}
	
	\begin{lemma}\label{rw:scalar-relations}
		The spectrum of \(K\) is \(\{-1,0,1\}\), and
		\begin{align}
			t_0+t_1+t_2&=0,\label{s3:hand:trace}\\
			t_0^2+t_1^2+t_2^2+2(p^2+q^2)&=2,\label{s3:hand:norm}\\
			t_0t_1t_2-t_2p^2-t_1q^2&=0.\label{s3:hand:det}
		\end{align}
		Moreover,
		\begin{equation}\label{s3:hand:Hidentities}
			H=t_0^2+p^2+q^2,\qquad t_1t_2=H-1,\qquad |t_i|\le1.
		\end{equation}
	\end{lemma}
	
	\begin{proof}
		Equation~\eqref{rw:complex-frame} gives \(Kz=iw\), \(Kw=-iz\),
		and \(K=0\) on the orthogonal complement of \(\Span_{\C}\{z,w\}\).
		Thus its eigenvalues are \(1,-1,0\), and
		\[
		K^2=zz^*+ww^*
		=\begin{pmatrix}
			t_0^2+p^2+q^2&-ip(t_0+t_1)&-iq(t_0+t_2)\\
			ip(t_0+t_1)&t_1^2+p^2&pq\\
			iq(t_0+t_2)&pq&t_2^2+q^2
		\end{pmatrix}.
		\]
		The trace of \(K\), the trace of \(K^2\), and the determinant of
		\(K\) give the first three identities.
		The entry numbered \(00\) of \(K^2=zz^*+ww^*\) gives the formula
		for \(H=|u_0|^2+|v_0|^2\).
		Using \(t_1+t_2=-t_0\), the trace identity becomes
		\[
		2=t_1^2+t_2^2-t_0^2+2H=-2t_1t_2+2H.
		\]
		Finally, \(|t_i|=|\langle Ke_i,e_i\rangle|\le\|K\|=1\), where
		\(e_0,e_1,e_2\) are the standard basis vectors of \(\C^3\).
	\end{proof}
	
	\subsection{Evaluation of the cubic coefficient}
	
	We retain the planar coordinates, area form \(\omega\), complex
	orthonormal frame \(z,w\), and scalar invariants \(p,q,t_i\)
	just defined. To exhibit the role of the real-coordinate weight,
	let
	\[
	S_\ell(r;u_0,u_1,u_2)=(\ell r;u_1,u_0-u_2,-u_1),
	\qquad \ell\in\R.
	\]
	Since \(r=s=0\), one has \(S_\ell x=S_3x\) and
	\(S_\ell y=S_3y\). Thus \(a,b,h,f_x,f_y\) in
	\eqref{s3:eq:coeffs} are unchanged when \(3\) is replaced by \(\ell\).
	Define
	\[
	\delta_\ell=-B(a,b)+\frac34B(S_\ell a,a)
	+\frac14B(S_\ell h,h)-B(S_\ell f_x,f_y).
	\]
	
	\begin{lemma}\label{s3:hand:deltaidentity}
		For every \((x,y)\in\mathcal A\) and every \(\ell\in\R\),
		\[
		\delta_\ell=4pq+\ell(4q^2+t_0t_1).
		\]
	\end{lemma}
	
	\begin{proof}
		We first compute the bracket contributions, and then evaluate the
		single determinant that remains.
		
		The planar versions of \eqref{s3:hand:cross} are
		\begin{equation}\label{s3:hand:omega-relations}
			\begin{gathered}
				\omega(u_i,v_j)+\omega(u_j,v_i)=0\quad(i\ne j),\\
				\sum_i\omega(u_i,v_i)=0,\qquad u_1\cdot v_2=u_2\cdot v_1.
			\end{gathered}
		\end{equation}
		Define the real quadratic expressions
		\[
		\begin{aligned}
			Q_x&=-|u_1|^2-u_0\cdot u_2+|u_2|^2,\\
			Q_y&=-|v_1|^2-v_0\cdot v_2+|v_2|^2,\\
			Q_{xy}&=-u_1\cdot v_1+u_2\cdot v_2
			-\frac12(u_0\cdot v_2+u_2\cdot v_0).
		\end{aligned}
		\]
		Substitution in \eqref{s3:eq:bracket}, using
		\eqref{s3:hand:omega-relations}, gives
		\begin{equation}\label{rw:planar-brackets}
			\begin{gathered}
				a=(-2q;0,2t_2n,0),\qquad b=(2p;-t_1n,0,2t_1n),\\
				h_r=4Q_{xy},\qquad h_0=0,\qquad
				h_1=4\omega(u_0,v_2)n,\\
				h_2=4\{\omega(u_0,v_1)+\omega(u_1,v_2)\}n,\\
				(f_x)_r=2Q_x,\qquad (f_x)_0=0,\qquad
				(f_x)_1=2\omega(u_0,u_2)n,\\
				(f_x)_2=2\{\omega(u_0,u_1)+\omega(u_1,u_2)\}n.
			\end{gathered}
		\end{equation}
		The expressions for \(f_y\) replace all \(u_i\) by \(v_i\).
		For example,
		\[
		\begin{aligned}
			\tfrac12[Sx,y]_r
			&=-u_1\cdot v_1-u_0\cdot v_2+u_2\cdot v_2,\\
			\tfrac12[x,Sy]_r
			&=u_2\cdot v_0-u_2\cdot v_2+u_1\cdot v_1,\\
			\tfrac12[Sx,y]_1
			&=\{\omega(u_1,v_1)+\omega(u_0,v_0)
			-\omega(u_2,v_0)\}n,\\
			\tfrac12[x,Sy]_1
			&=\{\omega(u_0,v_0)-\omega(u_0,v_2)
			+\omega(u_1,v_1)\}n.
		\end{aligned}
		\]
		Their sums and differences give \(a_r,h_r,a_1,h_1\) in
		\eqref{rw:planar-brackets}. The other blocks follow from the
		same formula \eqref{s3:eq:bracket}.
		
		The first two terms of \(\delta_\ell\) are now
		\[
		-B(a,b)=4pq,\qquad \frac34B(S_\ell a,a)=3\ell q^2.
		\]
		Indeed, \(a,b\) have disjoint nonzero imaginary blocks, and the
		imaginary blocks of \(S_\ell a\) are orthogonal to those of \(a\).
		
		The contributions of the imaginary-coordinate blocks to
		\(\frac14B(S_\ell h,h)-B(S_\ell f_x,f_y)\) add up to the real scalar
		\[
		\begin{aligned}
			&-8\omega(u_0,v_2)\{\omega(u_0,v_1)+\omega(u_1,v_2)\}\\
			&\quad+4\{\omega(u_0,u_1)+\omega(u_1,u_2)\}\omega(v_0,v_2)\\
			&\quad+4\omega(u_0,u_2)\{\omega(v_0,v_1)+\omega(v_1,v_2)\}.
		\end{aligned}
		\]
		This scalar vanishes. To verify the cancellation, use the planar identity
		\begin{equation}\label{s3:hand:plucker}
			\omega(a,b)\omega(c,d)
			=\omega(a,c)\omega(b,d)-\omega(a,d)\omega(b,c),
		\end{equation}
		which follows by expanding the four determinants. In particular,
		\[
		\begin{aligned}
			&\omega(u_0,u_1)\omega(v_0,v_2)
			+\omega(u_0,u_2)\omega(v_0,v_1)\\
			&\quad=\omega(u_0,v_0)\{\omega(u_1,v_2)+\omega(u_2,v_1)\}\\
			&\qquad-\omega(u_0,v_2)\omega(u_1,v_0)
			-\omega(u_0,v_1)\omega(u_2,v_0)\\
			&\quad=2\omega(u_0,v_2)\omega(u_0,v_1),
		\end{aligned}
		\]
		and
		\[
		\begin{aligned}
			&\omega(u_1,u_2)\omega(v_0,v_2)
			+\omega(u_0,u_2)\omega(v_1,v_2)\\
			&\quad=\{\omega(u_1,v_0)+\omega(u_0,v_1)\}\omega(u_2,v_2)\\
			&\qquad-\omega(u_1,v_2)\omega(u_2,v_0)
			-\omega(u_0,v_2)\omega(u_2,v_1)\\
			&\quad=2\omega(u_0,v_2)\omega(u_1,v_2).
		\end{aligned}
		\]
		The last equalities use \eqref{s3:hand:omega-relations}.
		Their sum proves the claimed cancellation.
		
		The real-coordinate contribution is
		\(4\ell Q_{xy}^2-4\ell Q_xQ_y\).
		Writing \(D_Q=Q_xQ_y-Q_{xy}^2\), we have proved
		\begin{equation}\label{s3:hand:deltadet}
			\delta_\ell=4pq+3\ell q^2-4\ell D_Q.
		\end{equation}
		
		It remains to evaluate \(D_Q\) from the scalar identities of
		Lemma~\ref{rw:scalar-relations}. Let
		\[
		Q_0=\begin{pmatrix}
			0&0&-1/2\\0&-1&0\\-1/2&0&1
		\end{pmatrix}.
		\]
		This fixed real symmetric \(3\times3\) matrix is distinct from the
		metric operator \(Q_\varepsilon\).
		The definitions give
		\[
		z^*Q_0z=Q_x,\qquad w^*Q_0w=Q_y,\qquad
		\Rea(z^*Q_0w)=Q_{xy}.
		\]
		Let \(D_{\C}\) be the determinant of the Hermitian compression of
		\(Q_0\) to the complex plane spanned by \(z,w\). Then
		\[
		D_{\C}
		=\det\begin{pmatrix}z^*Q_0z&z^*Q_0w\\w^*Q_0z&w^*Q_0w\end{pmatrix}
		=Q_xQ_y-|z^*Q_0w|^2,
		\]
		and therefore
		\begin{equation}\label{s3:hand:real-complex-det}
			D_Q=D_{\C}+\bigl(\im(z^*Q_0w)\bigr)^2.
		\end{equation}
		
		We compute the compression determinant through its complementary
		one-dimensional space. Complete \(z,w\) to a complex orthonormal
		basis \(z,w,\nu\), and put \(U=(z,w,\nu)\).
		The adjugate of a matrix is the transpose of its cofactor matrix;
		for an invertible matrix \(A\), it is \(\det(A)A^{-1}\).
		Since \(\det Q_0=1/4\), one has
		\[
		\operatorname{adj}(U^*Q_0U)=U^*\operatorname{adj}(Q_0)U.
		\]
		The last diagonal cofactor of \(U^*Q_0U\) is precisely \(D_{\C}\).
		Also \(\nu\nu^*=I-zz^*-ww^*=I-K^2\). Hence
		\begin{equation}\label{s3:hand:cofactor}
			D_{\C}=\nu^*\operatorname{adj}(Q_0)\nu
			=\tr\{\operatorname{adj}(Q_0)(I-K^2)\}.
		\end{equation}
		Taking minors, or multiplying by \(Q_0\), gives
		\[
		\operatorname{adj}(Q_0)
		=\begin{pmatrix}
			-1&0&-1/2\\0&-1/4&0\\-1/2&0&0
		\end{pmatrix},
		\qquad Q_0\operatorname{adj}(Q_0)=\frac14I.
		\]
		The entries numbered \(02\) and \(20\) of \(K^2\) are opposite
		purely imaginary numbers, and their contributions to the trace cancel.
		The two remaining diagonal terms give
		\[
		D_{\C}=-(1-H)-\frac14(1-t_1^2-p^2).
		\]
		
		Finally,
		\[
		z^*Q_0w=-\frac12\overline z_0w_2-\overline z_1w_1
		-\frac12\overline z_2w_0+\overline z_2w_2.
		\]
		Since \(\im(\overline z_iw_j)=\omega(u_i,v_j)\),
		the relation \(F_{02}=0\) and \(t_i=-2\omega(u_i,v_i)\) give
		\[
		\im(z^*Q_0w)
		=-\frac12\{\omega(u_0,v_2)+\omega(u_2,v_0)\}
		-\omega(u_1,v_1)+\omega(u_2,v_2)
		=\frac{t_1-t_2}{2}.
		\]
		Thus \eqref{s3:hand:real-complex-det} becomes
		\[
		4D_Q=-4(1-H)-1+t_1^2+p^2+(t_1-t_2)^2.
		\]
		Using successively \(p^2+q^2=H-t_0^2\),
		\(t_0=-t_1-t_2\), and \(t_1t_2=H-1\), we obtain
		\[
		\begin{aligned}
			4D_Q+q^2+t_0t_1
			&=-5+5H+2t_1^2-2t_1t_2+t_2^2-t_0^2+t_0t_1\\
			&=5(H-1-t_1t_2)=0.
		\end{aligned}
		\]
		Consequently \(4D_Q=-q^2-t_0t_1\).
		Substitution into \eqref{s3:hand:deltadet} proves the lemma.
	\end{proof}
	
	\subsection{A positive lower bound}
	
	The scalar reduction is now complete. Recall that
	\(p,q,t_0,t_1,t_2\) are the real entries of the Hermitian matrix
	\eqref{s3:hand:K}, and that \(H\in[1/4,1/2]\).
	We use only the identities of Lemma~\ref{rw:scalar-relations}
	and the formula of Lemma~\ref{s3:hand:deltaidentity}.
	
	\begin{proposition}\label{s3:hand:positive}
		Every pair \((x,y)\in\mathcal A\) satisfies
		\[
		\delta_3(x,y)\ge\frac{\sqrt5-1}{4}.
		\]
	\end{proposition}
	
	\begin{proof}
		Since \(t_1t_2=H-1\le-1/2\), both \(t_1,t_2\) are nonzero.
		Define
		\[
		\alpha=-\frac{t_1}{t_2}=\frac{t_1^2}{1-H}.
		\]
		The inequalities \(|t_1|\le1\) and \(1-H\ge1/2\) give
		\(0<\alpha\le2\).
		Dividing \eqref{s3:hand:det} by \(t_2\) gives
		\[
		t_0t_1=p^2-\alpha q^2.
		\]
		The cubic identity therefore becomes
		\begin{equation}\label{rw:positive-squares}
			\begin{aligned}
				\delta_3
				&=3p^2+4pq+3(4-\alpha)q^2\\
				&=2(p^2+q^2)+(p+2q)^2+3(2-\alpha)q^2\\
				&\ge2(p^2+q^2).
			\end{aligned}
		\end{equation}
		
		Set \(L=p^2+q^2\ge0\).
		The determinant identity, \(|t_1|,|t_2|\le1\), and
		\(1-H\ge1/2\) also give
		\[
		|t_0|
		=\frac{|t_2p^2+t_1q^2|}{1-H}
		\le2L.
		\]
		Now the lower bound for \(H\) is used:
		\[
		\frac14\le H=t_0^2+L\le4L^2+L.
		\]
		The nonnegative solution of this quadratic inequality satisfies
		\[
		L\ge\frac{\sqrt5-1}{8}.
		\]
		Combining this with \eqref{rw:positive-squares} proves the result.
	\end{proof}
	
	\begin{proof}[Proof of Theorem~\ref{s3:thm:critical}]
		Let \((g,x,y)\in\mathcal Z\).
		It is a horizontal orthonormal commuting triple and also satisfies
		\(\zeta(x,y)=0\).
		Proposition~\ref{rw:moment-bound} gives \(1/4\le H\le1/2\),
		so \((x,y)\in\mathcal A\).
		Proposition~\ref{s3:hand:positive} therefore yields the asserted
		bound on \(\delta_3\).
	\end{proof}
	
	\section{One parameter interval for all planes}
	\label{rw:completion}
	
	We combine the two estimates stated in Section~\ref{sec:roadmap}.
	The space \(\mathcal I\) consists of all undeformed horizontal
	\(B\)-orthonormal frames together with their group points.
	The continuous functions
	\[
	c(g,x,y)=[x,y],\qquad
	\zeta(g,x,y)=D^*([S_3x,y]+[x,S_3y]),\qquad
	\delta_3(g,x,y)
	\]
	are defined by \eqref{s3:eq:coeffs}--\eqref{s3:eq:delta};
	their values happen not to depend on \(g\).
	Let \(a_0,C_0,\varepsilon_1>0\) be the constants from
	Proposition~\ref{rw:uniform-proposition}, including the bound
	\(|\delta_3|\le C_0\). Thus
	\[
	N_\varepsilon\ge
	a_0(|c|^2+\varepsilon^2|\zeta|^2)
	+\varepsilon^3\delta_3-C_0\varepsilon^4.
	\]
	
	Put
	\[
	m=\frac14,\qquad F=|c|^2+|\zeta|^2,\qquad
	\mathcal K=\{\xi\in\mathcal I:\delta_3(\xi)\le m/2\}.
	\]
	Here \(\xi\) denotes a triple \((g,x,y)\).
	The set \(\mathcal K\) is compact.
	Theorem~\ref{s3:thm:critical} gives
	\(\delta_3\ge(\sqrt5-1)/4>m\) on
	\(F^{-1}(0)=\mathcal Z\), so \(\mathcal K\cap F^{-1}(0)\)
	is empty. If \(\mathcal K\ne\varnothing\), it follows that
	\[
	\kappa=\min_{\mathcal K}F>0.
	\]
	Indeed, a zero minimum would be attained at a point of
	\(\mathcal K\cap F^{-1}(0)\).
	If \(\mathcal K=\varnothing\), set \(\kappa=1\).
	Define
	\begin{equation}\label{s3:eq:epsilon0}
		\varepsilon_0=
		\min\left\{\varepsilon_1,1,\frac{m}{4C_0},
		\frac{a_0\kappa}{4C_0}\right\}>0.
	\end{equation}
	
	For \(\xi\in\mathcal I\setminus\mathcal K\) and
	\(0<\varepsilon\le\varepsilon_0\),
	\[
	N_\varepsilon(\xi)
	\ge\frac m2\varepsilon^3-C_0\varepsilon^4
	\ge\frac m4\varepsilon^3>0.
	\]
	For \(\xi\in\mathcal K\), the inequalities \(\varepsilon\le1\),
	\(F\ge\kappa\), and \(|\delta_3|\le C_0\) give
	\[
	\begin{aligned}
		N_\varepsilon(\xi)
		&\ge a_0\varepsilon^2F-C_0\varepsilon^3-C_0\varepsilon^4\\
		&\ge\varepsilon^2(a_0\kappa-2C_0\varepsilon)\\
		&\ge\frac12a_0\kappa\varepsilon^2>0.
	\end{aligned}
	\]
	These regions cover \(\mathcal I\), proving positivity of every
	numerator on the same parameter interval.
	
	For completeness, the Gram matrix of \(Xg=Q_\varepsilon xg\),
	\(Yg=Q_\varepsilon yg\) is
	\[
	\begin{pmatrix}
		B(Q_\varepsilon x,x)&B(Q_\varepsilon x,y)\\
		B(Q_\varepsilon x,y)&B(Q_\varepsilon y,y)
	\end{pmatrix}.
	\]
	Because \(x,y\) are \(B\)-orthonormal and \(Q_\varepsilon\ge I/2\),
	this matrix is at least \(I_2/2\); its determinant is therefore at
	least \(1/4\). The submersion preserves this Gram matrix on horizontal
	vectors. Dividing by its positive determinant proves positivity of
	sectional curvature.
	
	\begin{theorem}\label{s3:thm:positive-final}
		Let \(\Sigma=\Sp(2)/\Sp(1)\) be the Gromoll--Meyer quotient for
		the action
		\[
		k\star g=\diag(k,k)g\diag(k,1)^{-1}.
		\]
		The metric
		\[
		\widetilde g_{\varepsilon,g}(Ug,Vg)
		=B\bigl((I+\varepsilon S_3)^{-1}U,V\bigr),
		\qquad
		S_3(r;u_0,u_1,u_2)=(3r;u_1,u_0-u_2,-u_1),
		\]
		induces a smooth quotient metric \(g_\varepsilon\) satisfying
		\(\sec(g_\varepsilon)>0\) for every
		\(0<\varepsilon\le\varepsilon_0\), with \(\varepsilon_0\) as in
		\eqref{s3:eq:epsilon0}.
	\end{theorem}
	
	\begin{proof}
		Positive definiteness and descent were established in
		Section~\ref{sec}.
		Every deformed horizontal plane is represented by a triple in
		\(\mathcal I\), as proved following \eqref{rw:moving-frame}.
		The preceding argument gives a positive numerator and positive area
		denominator for all such triples on the interval
		\eqref{s3:eq:epsilon0}. This proves the theorem and completes the proof
		of Theorem~\ref{thm:main}.
	\end{proof}
	
	The number \(\kappa\) is obtained by compactness. Its positivity
	suffices for the existence of the common parameter interval; a
	numerical endpoint would require a quantitative lower bound for
	\(F\) on \(\mathcal K\).
	
	\bibliography{exotic-sphere}

\end{document}